\documentclass[10pt,a4paper]{article} 
\usepackage[T1]{fontenc}
\usepackage[utf8]{inputenc}
\usepackage{color}
\usepackage{amsmath,amssymb,amsthm,amsfonts,amstext,amsbsy,amscd}
\usepackage{breqn}
\usepackage{hyperref}
\usepackage{mathrsfs}
\usepackage{enumerate}
\usepackage{amsthm}

\title{Wasserstein and total variation distance\\[2mm] between marginals of Lévy processes} 
\author{
  Ester~Mariucci\footnote{Otto von Guericke Universität Magdeburg, Germany.
    E-mail: \href{mailto:mariucci@ovgu.de}{mariucci@ovgu.de}. This research was supported by the Federal Ministry for Education and Research through the Sponsorship provided by the Alexander von Humboldt Foundation and partially funded by the Deutsche Forschungsgemeinschaft (DFG, German Research Foundation) – 314838170, GRK 2297 MathCoRe.}
  \and  Markus~Reiß\footnote{Humboldt-Universität zu Berlin, Germany. E-mail: \href{mailto:mreiss@math.hu-berlin.de}{mreiss@math.hu-berlin.de}}}

\newcommand{\p} {\ensuremath {\mathbb{P}}}
\newcommand{\E} {\ensuremath {\mathbb{E}}}
\newcommand{\C} {\ensuremath {\mathbb{C}}}
\newcommand{\N} {\ensuremath {\mathbb{N}}}
\newcommand{\R} {\ensuremath {\mathbb{R}}}

\newcommand{\Z} {\ensuremath {\mathbb{Z}}}
\newcommand{\I} {\ensuremath {\mathbb{I}}}
\newcommand{\F} {\ensuremath {\mathscr{F}}}

\newcommand{\X} {\ensuremath {\mathcal{X}}}

\newcommand{\Li} {\ensuremath {\mathscr{L}}}
\newcommand{\B} {\ensuremath {\mathscr{B}}}

\newcommand{\No} {\ensuremath {\mathcal{N}}}

\newcommand{\Wi} {\ensuremath {\mathcal{W}}}

\newcommand{\var} {\ensuremath {\textnormal{Var}}}

\theoremstyle{plain}
\newtheorem{theorem}{Theorem}
\newtheorem*{result*}{Result}
\newtheorem{lemma}{Lemma}
\newtheorem{proposition}{Proposition}
\newtheorem{corollary}{Corollary}

\theoremstyle{definition}
\newtheorem{definition}{Definition}
\newtheorem{remark}{Remark}
\newtheorem{example}{Example}
\begin{document}
\date{16 July 2018}
\maketitle

\begin{abstract}{We present  upper bounds for the Wasserstein distance of order $p$ between the marginals of Lévy processes, including Gaussian approximations for jumps of infinite activity. Using the convolution structure, we further derive upper bounds for the total variation distance between the marginals of Lévy processes. Connections to other metrics like Zolotarev and Toscani-Fourier distances are established. The theory is illustrated by concrete examples and an application to statistical lower bounds. }
\end{abstract}


\section{Introduction}

L\'evy processes form the prototype of continuous-time processes with a continuous diffusion and a jump part. In applications, there is a high interest to disentangle these parts based on discrete observations. While A\"it-Sahalia and Jacod \cite{AJ} among many others propose an asymptotically  (as the observation distances become smaller) consistent test on the presence of jumps for general semimartingale models, Neumann and Rei\ss\ \cite{NR} argue that, already inside the class of $\alpha$-stable processes with $\alpha\in(0,2]$, no uniformly consistent test exists. The subtle, but important, difference is the uniformity over the class of processes. Mathematically, the difference is that on the Skorokhod path space $D([0,T])$ $\alpha$-stable processes, $\alpha\in(0,2)$, induce laws singular to that of Brownian motion ($\alpha=2$), while their respective marginals at $t_k=kT/n$, $k=0,\ldots,n$ for $n$ fixed, have equivalent laws, which even converge  in total variation distance as $\alpha\to 2$ to those of Brownian motion. It is our aim here to shed some light on the geometry of the marginal laws of one-dimensional L\'evy processes and to quantify the distance of the marginal laws non-asymptotically as a function of the respective L\'evy characteristics $(b,\sigma^2,\nu)$. The marginals form, of course, infinitely divisible distributions, but we prefer here the process point of view which is sometimes more intuitive.

Let us recall  the fundamental result by Gnedenko and Kolmogorov \cite{GK}.

\begin{theorem}[\cite{GK}]
Marginals of L\'evy processes $X^n=(X^n_t)_{t\ge 0}$ with characteristics $(b_n,\sigma_n^2,\nu_n)$ converge weakly to marginals of a L\'evy process $X=(X_t)_{t\ge 0}$ with characteristics $(b,\sigma^2,\nu)$ if and only if
\[ b_n\to b \text{ and } \sigma_n^2\delta_0+(x^2\wedge 1)\nu_n(dx)\xrightarrow{w} \sigma^2\delta_0+(x^2\wedge 1)\nu(dx),\]
where $\delta_0$ is the Dirac measure in $0$ and $\xrightarrow{w}$ denotes weak convergence of finite measures.
\end{theorem}

As a particular example, consider the compound Poisson process with L\'evy measure $\frac{\delta_{-\varepsilon}+\delta_{\varepsilon}}{2\varepsilon^2}$ that has jumps of size $\varepsilon$ and $-\varepsilon$ both at intensity $\frac{1}{2\varepsilon^2}$. Then as $\varepsilon\downarrow 0$, the marginals converge to those of a standard Brownian motion, which can also be derived from Donsker's Theorem. Below, we shall be able to quantify this rate of convergence for general L\'evy processes in terms of the (stronger) $p$-Wasserstein distances $\Wi_p$. The derived Gaussian approximation of the small jump part relies on the fine analysis by Rio \cite{Rio09} of the approximation error  in Wasserstein distance for the central limit theorem. This is the subject of Theorem  \ref{teo:smalljumps}, of which the following is a simplified statement:
\begin{result*}
Let $X^S(\varepsilon)$ be a Lévy process with characteristics $(0, 0, \nu_\varepsilon)$ where $\nu_\varepsilon$ is a L\'evy measure with support in $[-\varepsilon,\varepsilon]$. Introducing $\bar\sigma^2(\varepsilon) = \int_{-\varepsilon}^ {\varepsilon} x^2 \nu_\varepsilon(dx)$, there exists a constant $C$ depending only on $p$ such that:
$$
\Wi_p\big(\Li(X_t^{S}(\varepsilon)),\No(0,t\bar\sigma^2(\varepsilon)) \big) \leq C\min\big(\sqrt t\bar\sigma(\varepsilon) ,\varepsilon\big) \leq C \varepsilon.
$$
\end{result*}

A Gaussian approximation of the small jumps of Lévy processes has already been employed, for example when simulating trajectories of Lévy processes with infinite Lévy measure (see e.g. \cite{tankov}).

The above result is actually an intermediate step for the more general Corollary \ref{res:tensorizationlevy}, that bounds the $p$-Wasserstein distance $\Wi_p$ in $\ell^r(\R^n)$ as follows:

\begin{result*}
 Let $X^j$, $j=1,2$, be two Lévy processes with characteristics $(b_j,\sigma_j^2,\nu_j)$, $j=1,2$. Then for all $\varepsilon\geq 0$, $T>0$ and $n\in\N$ we have
 \begin{align*}
  \Wi_p\Big(&(X_{kT/n}^1-X_{(k-1)T/n}^1)_{k=1}^n,(X_{kT/n}^2-X_{(k-1)T/n}^2)_{k=1}^n\Big)\\
  &\leq T n^{\frac{1}{r}-1} \big|b_1(\varepsilon)-b_2(\varepsilon)\big|
    + T^{1/2} n^{\frac{1}{r}-\frac12}\big|\sigma_1+\bar\sigma_1(\varepsilon)-\sigma_2-\bar\sigma_2(\varepsilon)\big|\\
&\quad  +
 C\sum_{j=1}^2\min\big(T^{1/2}n^{\frac{1}{r}-\frac12}\bar\sigma_j(\varepsilon),n^{\frac{1}{r}}\varepsilon\big)
 +n^{\frac{1}{r}}\Wi_p\big(X_{T/n}^{1,B}(\varepsilon),X_{T/n}^{2,B}(\varepsilon)\big),
 \end{align*}
 where $b_j(\varepsilon):=b_j-\int_{\varepsilon<|x|\leq 1}x\nu_j(dx)$, $\bar\sigma_j^2(\varepsilon):=\int_{|x|\leq\varepsilon}x^2\nu(dx)$ and $C$ is a constant depending only on $p$. The term $\Wi_p(X_{T/n}^{1,B}(\varepsilon),X_{T/n}^{2,B}(\varepsilon))$, involving the jumps larger $\varepsilon$, can be bounded as in Theorem \ref{teo:CPP}.
\end{result*}

Sometimes we can even obtain  bounds on the total variation distance, which for statistical purposes, especially testing, is particularly meaningful. The currently available bound in the literature is by Liese \cite{liese}.

\begin{theorem}[{\cite[Cor. 2.7]{liese}}]
For  L\'evy processes $X^1$ and $X^2$ with characteristics $(b_1,\sigma_1^2,\nu_1)$ and  $(b_2,\sigma_2^2,\nu_2)$, respectively, introduce the squared Hellinger distance of the L\'evy measures (put $\nu_0=\nu_1+\nu_2$):
\[  H^2(\nu_1,\nu_2):= \int_{\R} \bigg(\sqrt{\frac{d\nu_1}{d\nu_0}(x)}-\sqrt{\frac{d\nu_2}{d\nu_0}(x)}\bigg)^2\nu_0(dx).
\]
Then the total variation distance between the laws of $X_t^1$ and $X_t^2$ is bounded as:
\begin{align*}&\|\Li(X_t^1)-\Li(X_t^2)\|_{TV}\\
&\leq 2\sqrt{1-\Big(1-\frac12 H^2\Big(\No(\tilde b_1t,\sigma_1^2t),\No(\tilde b_2t,\sigma_2^2t)\Big)\Big)^2\exp\Big(- t H^2(\nu_1,\nu_2)\Big)}
\end{align*}
with $\tilde b_1=b_1-\int_{-1}^1x\nu_1(dx)$, $\tilde b_2=b_2-\int_{-1}^1x\nu_2(dx)$.
\end{theorem}

Note that the bound is very loose or  even trivial in the case $\nu_2=0$ and $\lambda_1=\nu_1(\R)>1/t$ because then $tH^2(\nu_1,\nu_2)=t\lambda_1>1$. So, this bound does not allow to deduce a total variation approximation of Brownian motion by jump processes of infinite jump activity like $\alpha$-stable processes with $\alpha\uparrow 2$. In fact, for pure jump L\'evy processes these bounds are analogous to the bounds by Mémin and Shiryayev \cite{MS} in the path space $D([0,T])$, where pure jump processes and Brownian motion have singular laws (for other results on distances on $D([0,T])$ see e.g. \cite{etore14, g13,jacod,Ku}). Our main idea is to use the convolutional structure of the laws to transfer bounds from Wasserstein to total variation distance. This strategy is implemented for L\'evy processes with a non-zero Gaussian component (but without any restriction on the Lévy measures, which can be infinite, and even with infinite variation) and yields Theorem \ref{th:MainTV}:

\begin{result*}
For  L\'evy processes $X^1$ and $X^2$ with characteristics $(b_j,\sigma_j^2,\nu_j)$ and $\sigma_j>0$, $j=1,2$, we have for all $t>0$, $\varepsilon\in [0,1]$:
 \begin{align*}
\big\|\Li(X_t^1)-\Li(X_t^2)&\big\|_{TV} \leq  \frac{\sqrt{\frac{t}{2\pi}} \Big|b_1(\varepsilon)-b_2(\varepsilon)\Big|+\sqrt 2\Big|\sqrt{\sigma_1^2+\bar\sigma_1^2(\varepsilon)}-\sqrt{\sigma_2^2+\bar\sigma_2^2(\varepsilon)}\Big|}{\sqrt{\sigma_1^2+\bar\sigma_1^2(\varepsilon)}\vee \sqrt{\sigma_2^2+\bar\sigma_2^2(\varepsilon)}}\\
&\quad+ \sum_{j=1}^2\sqrt{\frac{2}{\pi t \sigma_j^2}}
 \min\Big(2\sqrt{t\bar \sigma_j^2(\varepsilon)}, \frac{\varepsilon}{2}\Big)\\
 &\quad
 +t\big|\lambda_1(\varepsilon)-\lambda_2(\varepsilon)\big|+t\big(\lambda_1(\varepsilon)\wedge\lambda_2(\varepsilon)\big)
 \bigg\|\frac{\nu_1^\varepsilon}{\lambda_1(\varepsilon)}-\frac{\nu_2^\varepsilon}{\lambda_2(\varepsilon)}\bigg\|_{TV},
 \end{align*}
 with the above notation, $\nu_j^\varepsilon=\nu_j(\cdot\setminus(-\varepsilon,\varepsilon))$ and $\lambda_j(\varepsilon)=\nu_j^\varepsilon(\R)$.
\end{result*}

The results proven in this paper provide further insight in the geometry of the space of discretely observed Lévy processes. At the same time,  their nonasymptotic character finds fruitful applications in nonparametric statistics, when proving general lower bounds in a minimax sense.  The technology is shown at work in Section \ref{subsec:JR}, making the original proof by Jacod and Rei\ss\ \cite{JR} for volatility estimation under high activity jumps simpler and much more transparent.

The results are stated in dimension one. After the first version of this paper was completed, however, new results for non-asymptotic multidimensional central limit theorem in Wasserstein distances have appeared (see e.g. \cite{TB}). Since a (special form of a) central limit theorem was the main technical tool in our proof of Theorem \ref{teo:smalljumps}, this makes a multidimensional extension of our findings a promising future research direction that seems worth investigating. Another potentially fruitful line of research would be to go beyond the independence structure of the increments and consider the general framework of semimartingales. Lévy processes are the basic building blocks for these more general processes and it is common to use this easier setting as a first step towards a more general proof; however, the techniques that were used in this paper heavily depend on the independence structure and do not directly extend to this more general framework.

The paper is organized as follows. In Section 2 we review basic properties of the Wasserstein distances and discuss their relationship with the Zolotarev and Toscani-Fourier distances. Then we recall the main non-asymptotic bounds for the Wasserstein distances in the CLT and introduce L\'evy processes.
Section 3 derives bounds between marginals of L\'evy processes in Wasserstein distance. The main focus is on the small jump part, which is treated in Theorem \ref{teo:smalljumps} and for which the tightness of the bounds is discussed in detail, first for concrete examples and then more generally using a lower bound via the Toscani-Fourier distance. Main results are presented in Section 3.3. Section 4 introduces  properties of the total variation distance and then shows how bounds in Wasserstein or Toscani-Fourier distance transfer under convolution to total variation bounds, see e.g. Proposition \ref{prop:TV} and Proposition \ref{prop:CTMR}. For Gaussian convolutions the different bounds are first compared and then applied to the marginals of L\'evy processes. Section 5 is devoted to the application of the total variation bounds for proving the minimax-optimality of integrated volatility estimators in the presence of jumps proposed in \cite{JR}.

\section{Preliminaries}
\subsection{The Wasserstein distances}
Let $(\mathcal X,d)$ be a Polish metric space. Given $p\in[1,\infty)$, let $\mathcal P_p(\X)$ denote the space of all Borel probability measures
$\mu$ on $\X$ such that the moment bound
\begin{equation*}
 \E_\mu[d(X,x_0)^p]<\infty
\end{equation*}
holds for some (and hence all) $x_0\in \X$.
\begin{definition}
Given $p\geq 1$, for any two probability measures $\mu,\nu\in\mathcal P_p(\X)$, the \emph{Wasserstein distance of order $p$}
between $\mu$ and $\nu$ is defined by
\begin{equation}\label{eq:defw}
\Wi_p(\mu,\nu)=\inf\Big\{\big[\E [d(X',Y')^p]\big]^{\frac{1}{p}},\ \Li(X')=\mu, \ \Li(Y')=\nu\Big\},
\end{equation}
where the infimum is taken over all random variables $X'$ and $Y'$ having laws $\mu$ and $\nu$, respectively. We abbreviate $\Wi_p(X,Y)=\Wi_p(\Li(X),\Li(Y))$ for random variables $X,Y$ with laws $\Li(X),\Li(Y)\in\mathcal P_p(\X)$.
\end{definition}
%\begin{notation}
%If the random variables $X$ and $Y$ have distributions $\mu$ and $\nu$, we shall sometimes write $\Wi_p(X,Y)$ instead of $\Wi_p(\mu,\nu)$. The same convention will be adopted for every other distance introduced in the sequel.
%\end{notation}

The following lemma introduces some properties of the Wasserstein distances that we will use throughout the paper. For a proof, the reader is
referred to \cite{villani09}, Chapter 6.
\begin{lemma}\label{lemma:w1}
 The Wasserstein distances have the following properties:
 \begin{enumerate}[(1)]
  \item For all $p\geq 1$, $\Wi_p(\cdot,\cdot)$ is a metric on $\mathcal P_p(\X)$.
  \item If $1\leq p\leq q$, then $\mathcal P_q(\X)\subseteq\mathcal P_p(\X)$, and $\Wi_p(\mu,\nu)\leq \Wi_q(\mu,\nu)$ for every
  $\mu,\nu\in\mathcal P_q(\X)$.
  \item Given a sequence $(\mu_n)_{n\geq 1}$ and a probability measure $\mu$ in $\mathcal P_p(\X)$
  $$\lim_{n\to\infty}\Wi_p(\mu_n,\mu)=0$$
  if and only if $(\mu_n)_{n\geq 1}$ converges to $\mu$ weakly and for some (and hence all) $x_0\in\X$
   $$\lim_{n\to\infty}\int_{\X}d(x,x_0)^p\mu_n(dx)=\int_{\X}d(x,x_0)^p\mu(dx).$$
   \item \label{eq:coupling} The infimum in \eqref{eq:defw} is actually a minimum; i.e., there exists a pair $(X^*,Y^*)$ of jointly distributed $\X$-valued random
   variables with $\Li(X^*)=\mu$ and $\Li(Y^*)=\nu$, such that
   \begin{equation*}
\Wi_p(\mu,\nu)^p=\E[d(X^*,Y^*)^p].
   \end{equation*}
  \end{enumerate}
\end{lemma}

Following the terminology used in \cite{zolbook}, we can say that the Wasserstein distances are \emph{ideal metrics} since they possess the following
two properties.

\begin{lemma}\label{lemma:proprieta}
 Let $\X$ be a separable Banach space. For any three $\X$-valued random variables $X,Y, Z$, with $Z$ independent of $X$ and $Y$, the inequality
 $$\Wi_p(X+Z,Y+Z)\leq\Wi_p(X,Y)$$
 holds. Furthermore, for any real constant $c$, we have
 \begin{equation}\label{eq:hom}
\Wi_p(cX,cY)=|c|\Wi_p(X,Y).
 \end{equation}
 \end{lemma}
\begin{proof}
 Lemma \ref{lemma:w1} guarantees the existence of two random variables $X^*,Y^*$, independent of $Z$, such that
 $$\Wi_p(X,Y)=\big(\E[ d(X^*,Y^*)^p]\big)^{1/p}.$$ We have:
 \begin{align*}
  \Wi_p(X+Z,Y+Z)&\leq\big(\E\big[d(X^*+Z,Y^*+Z)^p\big]\big)^{\frac{1}{p}}=\big(\E\big[d(X^*,Y^*)^p\big]\big)^{\frac{1}{p}}=\Wi_p(X,Y).
 \end{align*}
 The equality \eqref{eq:hom} follows by homogeneity of the expectation.
\end{proof}

An immediate corollary of Lemma \ref{lemma:proprieta} is the subadditivity of the metric $\Wi_p$ under independence (or equivalently under convolution of laws).

\begin{corollary}\label{cor:subadd}
If $X_1,\dots,X_n$ are independent random variables as well as $Y_1,\dots,Y_n$, then
$$\mathcal W_p(X_1+\dots +X_n, Y_1+\dots +Y_n)\leq \sum_{i=1}^n\mathcal W_p(X_i,Y_i).$$
\end{corollary}

\begin{proof}
By induction, it suffices to prove the case $n=2$.
Let $\tilde X_2$ be a random variable equal in law to $X_2$ and independent of $Y_1$ and of $X_1$. By means of Lemma \ref{lemma:proprieta} we have
\begin{align*}
&\mathcal W_p(Y_1+\tilde X_2, Y_1+Y_2)\leq \mathcal W_p(\tilde X_2,Y_2)=\mathcal W_p(X_2,Y_2),\\
&\mathcal W_p(X_1+ X_2, Y_1+\tilde X_2)\leq \mathcal W_p(X_1, Y_1).
\end{align*}
Hence, by triangle inequality $\mathcal W_p(X_1+X_2, Y_1+Y_2)\leq \mathcal W_p(X_1,Y_1)+\mathcal W_p(X_2,Y_2)$ follows.
%$$\mathcal W_p(X_1+\dots +X_n, Y_1+\dots +Y_n)\leq \mathcal W_p(X_n,Y_n)+\mathcal W_p(X_1+\dots+X_{n-1}, Y_1+\dots+Y_{n-1}).$$
%By recurrence we thus get
%$$\mathcal W_p(X_1+\dots +X_n, Y_1+\dots +Y_n)\leq \sum_{i=1}^n\mathcal W_p(X_i,Y_i).$$
\end{proof}
A useful property of the Wasserstein distances is their good behaviour with respect to products of measures.

\begin{lemma}[Tensorisation]\label{property:tensorization}
Let $(\X,d)$ be $\R^n$ endowed with the distance
$d(x,y)=(\sum_{i=1}^n{|x_i-y_i|}^r)^{1/r}$, $r \geq 1$, for all $x=(x_1,\dots,x_n)$, $y=(y_1,\dots,y_n)$ and let $\mu=\bigotimes_{i=1}^n\mu_i$
and $\nu=\bigotimes_{i=1}^n\nu_i$ be two product measures on $\R^n$.
Then,
\begin{equation*}
\Wi_p(\mu, \nu)^p \leq \max(n^{\frac{p}{r}-1},1)\sum_{i=1}^n \Wi_p(\mu_i, \nu_i)^p.
\end{equation*}
In the special case where $\mu_1=\dots=\mu_n$ and $\nu_1=\dots=\nu_n$, the following stricter inequality holds:
$$
\Wi_p(\mu, \nu)^p \leq n^{\frac{p}{r}} \Wi_p(\mu_1,\nu_1)^p.
$$
\end{lemma}

\begin{proof}
 Thanks to Point \eqref{eq:coupling} in Lemma \ref{lemma:w1}, we can always find two  random vectors $X^{*,n}=(X_1^*,\dots,X_n^*)$,
 $Y^{*,n}=(Y_1^*,\dots,Y_n^*)$
 with independent coordinates such that $\mu_i=\Li(X_i^*)$, $\nu_i=\Li(Y_i^*)$ and
 $\Wi_p(\mu_i,\nu_i)=\E[|X_i^*-Y_i^*|^p]^{1/p}$. In particular, we have:
 \begin{align}\label{eq:prodotto}
  \Wi_p(\mu,\nu)^p&\leq\E\big[ d((X_i^*)_i,(Y_i^*)_i)^p\big]=
  \E \bigg[\bigg(\sum_{i=1}^n|X_i^*-Y_i^*|^r\bigg)^{p/r}\bigg].
 \end{align}
 If $p\geq r$, by means of the elementary inequality $(z_1+\dots +z_n)^q\leq n^{q-1}(z_1^q+\dots+z_n^q)$, $q\geq 1$, we deduce from \eqref{eq:prodotto} that
 $$\Wi_p(\mu,\nu)^p\leq n^{p/r-1}\sum_{i=1}^n\E[|X_i^*-Y_i^*|^p]=n^{p/r-1}\sum_{i=1}^n\Wi_p(\mu_i,\nu_i)^p.$$
 Similarly, if $p<r$, the proof follows by the inequality $(|z_1|+\dots+|z_n|)^{1/q}\leq|z_1|^{1/q}+\dots+|z_n|^{1/q}$, $q\geq 1$.

 In the case where $\mu_1=\dots=\mu_n$ and $\nu_1=\dots=\nu_n$, one may choose $X_1^*=\dots=X_n^*$ and $Y_1^*=\dots=Y_n^*$. The conclusion readily follows.
 \end{proof}
 The distance $\Wi_1$ is commonly called the \emph{Kantorovich-Rubinstein distance} and it can be characterized in many different ways.
Some useful properties of the distance $\Wi_1$ are the following.

\begin{proposition}\label{pro:lip}[See \cite{GS}]
 Let $X$ and $Y$ be integrable real random variables. Denote by $\mu$ and $\nu$ their laws and by $F$ and $G$ their cumulative distribution functions, respectively. Then the following characterizations of the Wasserstein distance of order $1$ hold:
 \begin{enumerate}
  \item $\displaystyle{\Wi_1(X,Y)=\int_{\R}|F(x)-G(x)|dx},$
  \item $\displaystyle{\Wi_1(X,Y)=\int_0^1|F^{-1}(t)-G^{-1}(t)|dt},$
  \item $\Wi_1(X,Y)=\sup_{\|\psi\|_{\textnormal{Lip}}\leq 1}\bigg(\int_\R\psi d\mu-\int_\R\psi d\nu\bigg),$
  the supremum being taken over all $\psi$ satisfying the Lipschitz condition $|\psi(x)-\psi(y)|\leq |x-y|$, for all $x,y\in\R$. This property is generally called Kantorovich-Rubinstein formula.
 \end{enumerate}
\end{proposition}

 %Finally, the Wasserstein distances, also enjoy a variational representation.
 %\begin{theorem}[Kantorovich duality; see e.g. \cite{villani}, Theorem 5.9]\label{teo:duality}
  %Let $(\X,d)$ be a Polish metric space and $\mu,\nu\in \mathcal P_p(\X)$. Then
  %\begin{equation*}
   %\Wi_p(\mu,\nu)^p=\sup_{f\in C_b^{Lip}(\X)} \bigg(\int_\X f^*d\mu-\int_{\X}fd\nu\bigg),
  %\end{equation*}
%where $f^*(x):=\inf_{y\in\X}(f(y)+d(x,y))$ and $C_b^{Lip}$ is the class of bounded Lipschitz functions.
 %\end{theorem}
\subsection{Wasserstein, Zolotarev and Toscani-Fourier distances}

Let $\mu$, $\nu$ be two probability measures on $\R$ endowed with the distance $d(x,y)=|x-y|$, $x,y\in\R$.
 Writing $p>0$ as $p=m+\alpha$ with $m\in\N_0$  and $0<\alpha\leq 1$, denote by $\F_p$ the H\"older class of
 real-valued bounded functions $f$ on $\R$ which are $m$-times differentiable with
 $$\big|f^{(m)}(x)-f^{(m)}(y)\big|\leq |x-y|^\alpha.$$
\begin{definition}
  The \emph{Zolotarev distance} $Z_p$ between $\mu$ and $\nu$ is defined by
 \begin{equation*}
Z_p(\mu,\nu)=\sup_{f\in\F_p}\bigg(\int_\R fd\mu-\int_\R fd\nu\bigg).
 \end{equation*}
\end{definition}
\begin{remark}
 It is easy to see that the functional $Z_p$ is a metric. For $p=0$ the metric $Z_p$ is defined by the relation $Z_0=\lim_{p\to0}Z_p$ and
 $\F_0$ is the set of Borel functions satisfying the condition $|f(x)-f(y)|\leq \I_{x\neq y}$. Thanks to the characterisation of the total variation
 given in Property \ref{property:tv} below, it follows that $Z_0(\mu,\nu)=\|\mu-\nu\|_{TV}$. Also, by means of the Kantorovich-Rubinstein formula, recalled in Property \ref{pro:lip}, we have
 $Z_1(\mu,\nu)=\Wi_1(\mu,\nu).$
\end{remark}

The following result shows that the Wasserstein distance of order $p$ is bounded by the $p$-th root of the Zolotarev distance $Z_p$. This fact, together with Theorem \ref{teo:zolotarevCPP} below, will be a useful tool to control the Wasserstein distances between the increments of compound Poisson processes.

\begin{theorem}[See \cite{Rio09}, Theorem 3.1]\label{teo:Rio09}
 For any $p\geq 1$ there exists a positive constant $c_p$ such that for any pair $(\mu,\nu)$ of laws on the real line with finite absolute moments
 of order $p$
 $$\big(\Wi_p(\mu,\nu)\big)^p\leq c_pZ_p(\mu,\nu).$$
\end{theorem}

\begin{theorem}\label{teo:zolotarevCPP}[See \cite{zolbook}, Theorem 1.4.3]
  Let $(X_i)_{i\geq 1}$ and $(Y_i)_{i\geq 1}$ be sequences of independent random variables and $N$ be an integer-valued random variable independent of
  the random variables from both sequences. Then,
  \begin{equation*}
   Z_p\bigg(\sum_{i=1}^N X_i, \sum_{i=1}^N Y_i\bigg)\leq \sum_{k=1}^{\infty}\p(N\geq k)Z_p(X_k,Y_k).
  \end{equation*}
 \end{theorem}

\begin{theorem}[See \cite{zolbook}, Theorem 1.4.2.]
 Let $X$ and $Y$ be integrable real random variables with laws $\mu$ and $\nu$, respectively. Then the following characterization of the Zolotarev distance holds: for any $p\geq 1$
 \begin{equation*}
  Z_p(X,Y)=\int \bigg|\int_{-\infty}^x \frac{(x-u)^{p-1}}{\Gamma(p)}(\mu-\nu)(du)\bigg|dx,
 \end{equation*}
 where $\Gamma$ denotes the Gamma function.
\end{theorem}

Let $P_1$ and $P_2$ be two probability measures on the real line. We will denote by $\varphi_1$ (resp. $\varphi_2$) the characteristic function of $P_1$ (resp. $P_2$), i.e.
$$\varphi_1(u)=\int_{\R} e^{iux}P_1(dx).$$
Also, denote by $\Li^b(\R)$ (resp. $\Li^b(\C)$) the class of real-valued (resp. complex-valued) bounded functions on $\R$ with Lipschitz norm bounded by $1$.
\begin{definition}\label{def:toscani}
For  $s>0$, the \emph{Toscani-Fourier distance of order} $s$, denoted by $T_s$, is defined as:
$$T_s(P_1,P_2)=\sup_{u\in \R\setminus\{0\}}\frac{|\varphi_1(u)-\varphi_2(u)|}{|u|^s}.$$
\end{definition}
The distance introduced in Definition \ref{def:toscani} first appeared in \cite{GT}, under the name ``Fourier-based metrics'', to study the trend to equilibrium for solutions of the space-homogeneous Boltzmann equation for Maxwellian molecules. After that, it has been used in several other works, and especially linked to the kinetic theory, see \cite{CT07} for an overview. In \cite{villani09}, $T_2$ is called the ``Toscani distance''.

\begin{proposition}\label{prop:toscani}
 For all $p\geq 1$
 $$\Wi_p(P_1,P_2)\geq \frac{1}{\sqrt 2}T_1(P_1,P_2).$$
\end{proposition}
\begin{proof}
 Thanks to Lemma \ref{lemma:w1} and Property \ref{pro:lip}
 \begin{align*}
   \Wi_p(P_1,P_2)&\geq \Wi_1(P_1,P_2)=\sup_{\psi \in \Li^b(\R)}\bigg(\int_\R\psi dP_1-\int_\R\psi dP_2\bigg)\\
                 &\geq \frac{1}{\sqrt 2}\sup_{\psi\in \Li^b(\C)}\bigg|\int_\R\psi dP_1-\int_\R\psi dP_2\bigg|.
 \end{align*}
For all $u\in \R\setminus\{0\}$, let us consider the function $\Psi_u(x)=\frac{e^{iux}}{u}$ and observe that the Lipschitz norm of $\Psi_u$ is $1$. It immediately follows that
$$\sup_{\psi\in \Li^b(\C)}\bigg|\int_\R\psi dP_1-\int_\R\psi dP_2\bigg|\geq \sup_{u\in\R\setminus\{0\}}\bigg|\int_\R\Psi_u dP_1-\int_\R\Psi_u dP_2\bigg|= T_1(P_1,P_2).$$
\end{proof}

\subsection{Wasserstein distances in the central limit theorem}\label{CLT}
The class of Wasserstein metrics proves to be very useful  in  estimating  the convergence rate in the central limit theorem. We recall some results.
Let $(Y_i)_{i\geq 1}$ be a sequence of centred i.i.d. random variables with finite and positive variance $\sigma^2$. We denote by
$\mu_n$ the law of $\frac{1}{\sqrt{n\sigma^2}}\sum_{i=1}^n Y_i$. %and by $\No(m,\Sigma^2)$ the law of a Gaussian random variable with mean $m$
%and variance $\Sigma^2$.
For i.i.d. centred random variables with finite absolute third moment, Esseen \cite{esseen} proved the following result.
\begin{theorem}\label{teo:zolotarev}[See e.g. \cite{petrov}, Theorem 16]
For any $n\geq 1$,
 $$\Wi_1\big(\mu_n,\No(0,1)\big)\leq \frac{1}{2\sqrt n} \frac{\E|Y_1|^3}{\big(\var[Y_1]\big)^{3/2}}.$$
The constant $\frac{1}{2}$ in this inequality cannot
be improved.
 \end{theorem}
A bound for the Wasserstein distances of order $r\in(1,2]$ is due to Rio \cite{Rio09}:
\begin{theorem}\label{teo:rio}[See \cite{Rio09}, Theorem 4.1]
 For any $n\geq 1$ and any $r\in(1,2]$, there exists some positive constant $C$ depending only on $r$ such that
 $$\Wi_r(\mu_n,\No(0,1))\leq C \frac{\Big(\E\big[|Y_1|^{r+2}\big]\Big)^{1/r}}{\sqrt n\big(\var[Y_1]\big)^{\frac{r+2}{2r}}} .$$
\end{theorem}
For $r>2$ and i.i.d. random variables with a finite absolute moment of order $r$, we have the following:
\begin{theorem}\label{teo:sak}[See \cite{sak}]
 For any $n\geq 1$ and  $r>2$, there exists some positive constant $C$, depending only on $r$, such that
 $$\Wi_r(\mu_n,\No(0,1))\leq C  \frac{\Big(\E\big[|Y_1|^{r}\big]\Big)^{1/r}}{\sqrt{\var[Y_1]}}n^{\frac{1}{r}-\frac{1}{2}}.$$
\end{theorem}
If one only assumes finite absolute moment of order $r$, this rate cannot be improved. In particular, under this assumption, the classical rate of convergence $\frac{1}{\sqrt n}$ cannot be recovered for $r>2$. For that reason, from now on, we will only focus on the case $r\in[1,2]$.% when dealing with the Gaussian approximation of small jumps of Lévy processes.

 \subsection{Lévy processes}\label{sec:notationlevy}

 Let us denote by $P_t^{(b,\sigma,\nu)}$ the marginal law at time $t\ge 0$ of a Lévy process $X$ with characteristics $(b,
\sigma^2,\nu)$, i.e. (see Theorem 8.1 in \cite{sato})
\begin{align*}
\E\big(e^{iuX_t}\big)&=\exp\bigg(t\bigg(iub-\frac{u^2\sigma^2}{2}+\int_\R\big(e^{iux}-1-iux\I_{|x|\leq 1}\big)
\nu(dx)\bigg)\bigg)\\
&=\exp\bigg(t\bigg(iub(\varepsilon)-\frac{u^2\sigma^2}{2}+\int_\R\big(e^{iux}-1-iux\I_{|x|\leq \varepsilon}\big)
\nu(dx)\bigg)\bigg),
\end{align*}
where $b(\varepsilon):=b-\int_{\varepsilon< |x|\leq 1}x\nu(dx)$, for all $\varepsilon\in(0,1]$. Equivalently,
$P_t^{(b,\sigma,\nu)}$ denotes the infinitely divisible law with characteristics $(bt,
\sigma^2t,\nu t)$.
$X$ can be characterised via the Lévy-Itô decomposition (see \cite{sato}), that is via a canonical representation with independent components
$X=X^{(1)}+X^{(2)}(\varepsilon)+X^{S}(\varepsilon)+X^{B}(\varepsilon)$: For all $\varepsilon\in(0,1]$
\begin{align*}
 X_t&=\sigma W_t+b(\varepsilon)t+ \lim_{\eta\to 0}\bigg(\sum_{0< s\leq t}\Delta X_s\I_{(\eta,\varepsilon]}(|\Delta X_s|)-t
 \big(b(\varepsilon)-b(\eta)\big)\bigg)\\
 &\quad +\sum_{0< s\leq t}\Delta X_s\I_{(\varepsilon,+\infty)}(|\Delta X_s|),
\end{align*}
where $W$ is a standard Brownian motion, $\Delta X_s:=X_s-\lim_{r\uparrow s}X_r$ is the jump at time $s$ of  $X$, $X^{S}(\varepsilon)$ is a pure jump martingale containing only small jumps and $X^{B}(\varepsilon)$ is a finite variation part containing jumps larger in absolute value than $\varepsilon$. Thus
$X^{B}(\varepsilon)$ is a compound Poisson process with intensity $\lambda_\varepsilon:=\nu(\R\setminus(-\varepsilon,\varepsilon)
)$ and jump distribution $F_{\varepsilon}(dx)=\frac{\nu(dx)}{\nu(\R\setminus(-\varepsilon,\varepsilon))}\I_{(\varepsilon,+\infty)}(|x|)$.
In the following, sometimes we will write $X^{B}_t(\varepsilon)$ as $\sum_{i=1}^{N_t}Y_i$ where  $N$ is a Poisson process of intensity $\lambda_\varepsilon$ independent of the sequence $(Y_i)_{i\geq 0}$ of i.i.d. random variables having distribution $F_\varepsilon.$
Also, for a given Lévy process $X$ we define an auxiliary characteristic $\bar \sigma:\R_+\to\R$ capturing the variance induced by small jumps:
\begin{equation*}
 \bar \sigma^2(\varepsilon):=\int_{|x|\leq\varepsilon}x^2\nu(dx).
\end{equation*}

 \section{Wasserstein distances for Lévy processes}\label{sec:uppb}
  Let $X^j$, $j=1,2$, be two Lévy processes with characteristics  $(b_j,\sigma_j^2,\nu_j) $, $j=1,2$. As we will see later, thanks to Corollary \ref{cor:subadd} and the Lévy-Itô decomposition, in order to control $\Wi_p(X_t^1,X_t^2)$ it is enough to separately control the Wasserstein distances between two Gaussian random variables as well as  $\Wi_p\big(\Li(X_{t}^{j,S}(\varepsilon)),\No\big(0,t\bar \sigma_j^2(\varepsilon)\big)\big)$ and $\Wi_p\big(X_{t}^{1,B}(\varepsilon), X_{t}^{2,B}(\varepsilon)\big)$. A bound for the Wasserstein distances between Gaussian distributions is given by:
 \begin{lemma}\label{lemma:gaussiane}[See \cite{Givens1984}, Prop. 7]
 $$\Wi_2(\No(m_1,\sigma_1^2),\No(m_2,\sigma_2^2))=\sqrt{(m_1-m_2)^2+(\sigma_1-\sigma_2)^2}.$$
\end{lemma}

Upper bounds for $\Wi_p\big(\Li(X_{t}^{j,S}(\varepsilon)),\No\big(0,t\bar \sigma_j^2(\varepsilon)\big)\big)$ and $\Wi_p\big(X_{t}^{1,B}(\varepsilon), X_{t}^{2,B}(\varepsilon)\big)$ will be the subject of Sections \ref{sss:smalljumps} and \ref{sss:bigjumps}, respectively.

\subsection{Distances between marginals of small jump Lévy processes }\label{sss:smalljumps}
Let $X$ be a Lévy process with Lévy measure $\nu$ and denote by $X^{S}(\varepsilon)$ the Lévy process associated with the small jumps of $X$,
following the notation introduced in Section \ref{sec:notationlevy}.

\begin{theorem}\label{teo:smalljumps}
For any $p\in[1,2]$, there exists a positive constant $C$ such that
 \begin{align}
\Wi_p\big(\Li\big(X_t^{S}(\varepsilon)\big),\No(0,t\bar\sigma(\varepsilon)^2) \big)&\leq C\min\bigg(\sqrt t\bar\sigma(\varepsilon) ,
\bigg(\frac{\int_{-\varepsilon}^{\varepsilon}|x|^{p+2}\nu(dx)}{\bar\sigma^2(\varepsilon)}\bigg)^{1/p}\bigg)\nonumber\\
&\leq C\min\Big(\sqrt t\bar\sigma(\varepsilon) ,\varepsilon\Big) \label{eq:sj}.
 \end{align}
In particular, for $p=1$ the bound is $\min(2\sqrt t\bar\sigma(\varepsilon),\frac12\varepsilon)$.
\end{theorem}

\begin{remark}
 The inequality
 $$\Wi_p\big(\Li\big(X_t^{S}(\varepsilon)\big),\No(0,t\bar\sigma^2(\varepsilon)) \big)\leq\Wi_2\big(\Li\big(X_t^{S}(\varepsilon)\big),\No(0,t\bar\sigma^2(\varepsilon)) \big)\leq 2 \sqrt t\bar\sigma(\varepsilon)$$
 is clear from the definition of $\Wi_2$, noting that $t\bar\sigma^2(\varepsilon)$ is the second moment of both arguments. The interest of Theorem \ref{teo:smalljumps} lies in the bound
 \begin{equation}\label{eq:sjbis}
 \Wi_p\big(\Li\big(X_t^{S}(\varepsilon)\big),\No(0,t\bar\sigma^2(\varepsilon)) \big)\leq 2\varepsilon,
 \end{equation}
which after renormalisation yields
 $$ \Wi_p\bigg(\Li\Big(\frac{X_t^{S}(\varepsilon)}{\sqrt t\bar\sigma(\varepsilon)}\Big),\No(0,1) \bigg)\leq \frac{C\varepsilon}{\sqrt t\bar\sigma(\varepsilon)}.
$$
Thus, not surprisingly in view of the central limit theorem, the Gaussian approximation is better as $t$ is large. Also whenever $\bar\sigma^2(\varepsilon)$ is much larger than $\varepsilon^2$, then a Gaussian approximation is valid, e.g. for $\alpha$-stable processes with $\alpha>0$ and $\varepsilon$ small, see Example \ref{ex:stable} below.
\end{remark}

\begin{remark}\label{rmk:ex}
The upper bound \eqref{eq:sj} gives  in general the right order. Indeed,
let us consider for $\varepsilon>0$ the Lévy measure $\nu_\varepsilon=\frac{\delta_{-\varepsilon}+\delta_\varepsilon}{2\varepsilon^2}$
and denote by
$Y(\varepsilon)$ the corresponding (centred) pure jump Lévy process, i.e. $Y_t(\varepsilon)=\varepsilon(N_t^1(\varepsilon)-N_t^2(\varepsilon))$ is the rescaled difference of two independent Poisson processes of intensity
$\lambda=\frac{1}{2\varepsilon^2}$ each. In particular, observe that
$\bar\sigma^2(\varepsilon)=1$ and $\int_{-\varepsilon}^{\varepsilon}|x|^3\nu_\varepsilon(dx)=\varepsilon$.

Let us develop the case $p=1$. Applying the scheme of proof proposed in \cite{Rio09}, see proof of
Theorem 5.1, we show that there exists a constant $K$ such that
$$\Wi_1\big(\Li( Y_t(\varepsilon)),\No(0,t)\big)\geq K\min(\sqrt t,\varepsilon).$$
To see that, we consider the cases where $t\leq \varepsilon^2$ and $t>\varepsilon^2$ separately.
\begin{itemize}
 \item $t\leq \varepsilon^2$: From the definition of the Wasserstein distance of order $1$ it follows that
 \begin{align*}
\Wi_1\big(\Li( Y_t(\varepsilon)),\No(0,t)\big)&\geq \E[|\No(0,t)|]\p(N_t^1(\varepsilon)+N_t^2(\varepsilon)=0)\\
&=\sqrt{\frac{2t}{\pi}}e^{-\frac{t}{\varepsilon^2}}\geq
\sqrt{\frac{2}{\pi}}\frac{1}{e}\sqrt t.
 \end{align*}
 \item $t\geq \varepsilon^2$: Again, by the definition of the Wasserstein distance of order $1$, we find that
 \begin{align*}
 \Wi_1\big(\Li( Y_t(\varepsilon)), \No(0,t)\big) &\geq \E\Big[\min_{n\in\Z}|\sqrt t N-n\varepsilon |\Big]\\
 & \geq \frac{\varepsilon}{4}\p\Big((\sqrt t/\varepsilon) N\in \bigcup_{n\in\Z} [n+1/4,n+3/4]\Big)
 \end{align*}
 with $N\sim \No(0,1)$. Since in this case $(\sqrt t/\varepsilon) N$ has variance at least one,
 there exists a constant $K$ such that
$$\Wi_1\big(\Li( Y_t(\varepsilon)),\No(0,t)\big)\geq K\varepsilon.$$
 \end{itemize}
 In the case $p\in(1,2]$ $\Wi_p$ is even larger than $\Wi_1$. For the case $p=2$ see also \cite{fournier}.
\end{remark}

\begin{example}\label{ex:stable}
Let us illustrate Theorem \ref{teo:smalljumps} for the class of $\alpha$-stable Lévy processes with a Lévy density proportional to $\frac{1}{|x|^{1+\alpha}}$, $\alpha\in[0,2)$. For all $\varepsilon\in (0,1]$, let us denote by $X^S(\varepsilon)$ the Lévy process describing the small jumps and by $\nu\I_{[-\varepsilon,\varepsilon]}$ its Lévy measure, i.e.
$$X^S(\varepsilon)\sim\big(0,0,\nu\I_{[-\varepsilon,\varepsilon]}\big),\quad \nu(dx)=\frac{C_\alpha}{|x|^{1+\alpha}}dx,$$
for some constant $C_\alpha$. In particular, we have
$$\bar\sigma^2(\varepsilon)=\int_{-\varepsilon}^\varepsilon x^2\nu(dx)=2C_\alpha\frac{\varepsilon^{2-\alpha}}{2-\alpha}.$$
Therefore an application of Theorem \ref{teo:smalljumps} guarantees the existence of a constant $C$, possibly depending on $p$ and $\alpha$, such that:
\begin{equation}\label{eq:alpha}
\Wi_p\Big(\Li\Big(\frac{X_t^S(\varepsilon)}{\bar\sigma(\varepsilon)}\Big),\No(0,t)\Big)\leq C\min\big(\sqrt t,\varepsilon^{\frac{\alpha}{2}}\big),\quad \forall t>0, \ \forall \varepsilon\in(0,1], \ \forall p\in[1,2].
\end{equation}
Equation \eqref{eq:alpha} validates the intuition that a Gaussian approximation of the small jumps is the better the more active the small jumps are. Indeed, the approximation in \eqref{eq:alpha} is better when $\alpha$ is larger.
\end{example}

Let us now prove Theorem \ref{teo:smalljumps}. For that we need to recall the following lemma:

\begin{lemma} [See \cite{Rusch2002}, Lemma 6.]\label{lemma:momenti}
Let $X$ be a Lévy process with Lévy measure $\nu$. If a Borel function $f:\R\to\R$ satisfies $\int_{|x|\geq 1}f(x)\nu(dx)<\infty$, $\lim_{x\to 0}\frac{f(x)}{x^{2}}=0$ and
 $f(x)(|x|^2\wedge1)^{-1}$ is bounded, then
 $$\lim_{t\to 0}\frac{1}{t}\E[f(X_t)]=\int f(x)\nu(dx).$$
\end{lemma}

\begin{proof}[Proof of Theorem \ref{teo:smalljumps}]
Let us introduce $n$ random variables defined by
$Y_j=\sqrt n(X^{S}_{tj/n}(\varepsilon)-X^{S}_{t(j-1)/n}(\varepsilon))$. The $Y_j$'s are i.i.d.
centred random variables
with variance equal to $t\bar\sigma^2(\varepsilon)$ and such that $X_{t}^{S}(\varepsilon)=\frac{1}{\sqrt n}\sum_{j=1}^n Y_j$.
An application of Theorems \ref{teo:rio} and \ref{teo:zolotarev} (using the fact that $Y_j$ has the same law as $\sqrt nX_{t/n}^{S}(\varepsilon)$ and the homogeneity
property of the Wasserstein distances stated in Lemma \ref{lemma:proprieta}) gives
\begin{align*}
 \Wi_1(\Li\big(X_{t}^{S}(\varepsilon)\big),\No(0,t\bar\sigma^2(\varepsilon))\big)\leq \frac{n\E[|X_{t/n}^{S}(\varepsilon)|^3]}
 {2t\bar\sigma^2(\varepsilon)}.
 %=\frac{\int_{-\varepsilon}^{\varepsilon}|x|^3\nu_i(dx)}{2\sigma^2_\varepsilon(\nu_i)}.
 \end{align*}
Let us now argue that
 $$\limsup_{t\to 0}\frac{\E[|X_t^S(\varepsilon)|^3]}{t}\leq \int_{|x|<\varepsilon}|x|^3\nu(dx).$$
Indeed, applying Lemma \ref{lemma:momenti} to the family $f(x)=f_R(x)=|x|^3 \I_{[-R, R]}(x)$ for $R>\varepsilon$, we deduce that
 $$\lim_{t\to 0}\frac{\E\big[|X_t^S(\varepsilon)|^3\I_{|X_t^S(\varepsilon)|\leq R}\big]}{t}=\int_{|x|<\varepsilon}|x|^3\nu(dx).$$
 Thus, using the fact that $\E\big[(X_t^S(\varepsilon))^4\big]=t\int_{|x|<\varepsilon}x^4 \nu(dx)+3t^2\bar\sigma^4(\varepsilon)$, we get
 \begin{align*}
 \E\big[|X_t^S(\varepsilon)|^3\big]&\le \E\big[|X_t^S(\varepsilon)|^3\I_{|X_t^S(\varepsilon)|\leq R}\big]+\E\big[(X_t^S(\varepsilon)^4/R)\I_{|X_t^S(\varepsilon)|> R}\big]\\
 &\leq \E\big[|X_t^S(\varepsilon)|^3\I_{|X_t^S(\varepsilon)|\leq R}\big]+ \frac{1}{R} \bigg(t  \int_{|x|<\varepsilon}x^4 \nu(dx)+3t^2\bar\sigma^4(\varepsilon)\bigg).
 \end{align*}
 Therefore, for any $R>\varepsilon$,
 $$\limsup_{t\to 0}\frac{\E[|X_t^S(\varepsilon)|^3]}{t}\leq \int_{|x|<\varepsilon}|x|^3\nu(dx)+\frac{  \int_{|x|<\varepsilon}x^4 \nu(dx)}{R}.$$
 Taking the limit as $R\to\infty$, we conclude. % that $\limsup_{t\to 0}\frac{\E[|X_t^S(\varepsilon)|^3]}{t}\leq \int_{|x|<\varepsilon}|x|^3\nu(dx)$.
 It follows that
$$\Wi_1\big(\Li\big(X_{t}^{S}(\varepsilon)\big),\No(0,t\bar\sigma^2(\varepsilon))\big)\leq\limsup_{n\to\infty}\frac{n\E[|X_{t/n}^{S}(\varepsilon)|^3]}
 {2t\bar\sigma^2(\varepsilon)}\leq  \frac{\int_{-\varepsilon}^{\varepsilon}|x|^3\nu(dx)}{2\bar\sigma^2(\varepsilon)}.$$
Moreover, by definition of the Wasserstein distance of order $1$ and denoting by $N$ a centered Gaussian random variable with variance $t\bar\sigma^2(\varepsilon)$, we have
$$\Wi_1(\Li\big(X_{t}^{S}(\varepsilon)\big),\No(0,t\bar\sigma^2(\varepsilon)))\leq \E[|X_{t}^{S}(\varepsilon)|]+\E[|N|]
\leq 2\sqrt{t\bar\sigma^2(\varepsilon)}.$$
We  deduce
$$\Wi_1\big(\Li\big(X_{t}^{S}(\varepsilon)\big),\No(0,t\bar\sigma^2(\varepsilon))\big)\leq \min\bigg(2\sqrt{t\bar\sigma^2(\varepsilon)},
\frac{\int_{-\varepsilon}^{\varepsilon}|x|^3\nu(dx)}{2\bar\sigma^2(\varepsilon)}\bigg).$$
Similarly, by means of Theorem \ref{teo:rio}, for $p\in(1,2]$
\begin{align*}
 \Wi_p\big(\Li\big(X_{t}^{S}(\varepsilon)\big),\No(0,t\bar\sigma^2(\varepsilon))\big)&\leq \limsup_{n\to\infty}\frac{C}{\sqrt n}\bigg(\frac{\E\big[|\sqrt nX_{t/n}^{S}|^{p+2}\big]}
 {t\bar\sigma^2(\varepsilon)}\bigg)^{1/p}\\
 &\leq C\bigg(\frac{\int_{-\varepsilon}^{\varepsilon}|x|^{p+2}\nu(dx)}{\bar\sigma^2(\varepsilon)}\bigg)^{1/p}
\end{align*}
and also
$$\Wi_p\big(\Li\big(X_{t}^{S}(\varepsilon)\big),\No(0,t\bar\sigma^2(\varepsilon))\big)\leq \Big(\E\big[|X_{t}^{S}(\varepsilon)|^p\big]\Big)^{1/p}+\Big(\E\big[
|N|^p\big]\Big)^{1/p}
\leq 2\sqrt{t\bar\sigma^2(\varepsilon)}.$$
The upper bound \eqref{eq:sj} follows by the fact that $\frac{\int_{-\varepsilon}^{\varepsilon}|x|^{p+2}\nu_i(dx)}{\bar\sigma_i^2(\varepsilon)}\leq \varepsilon^p$.
\end{proof}

Theorem \ref{teo:smalljumps} can be used to bound the Wasserstein distances between the increments of the small jumps of two Lévy processes.
\begin{corollary}\label{cor:smalljumps}
 For all $\varepsilon\in(0,1]$ let $X^j(\varepsilon)\sim(-\int_{\varepsilon<|x|\leq 1}x\nu_j(dx),\sigma_j^2,\nu_j\I_{[-\varepsilon,\varepsilon]})$ be two Lévy processes with $\bar \sigma_j^2(\varepsilon)\neq 0$ and $\sigma_j\geq 0$, $j=1,2.$ Then, for all $p\in[1,2]$, there exists a  constant $C$, only depending on $p$, such that
 \begin{align*}
\Wi_p\big(X_t^{1}(\varepsilon),X_t^{2}(\varepsilon)\big)&\leq\sum_{j=1}^2C\min\bigg(\sqrt t\bar\sigma_j(\varepsilon) ,
\bigg(\frac{\int_{-\varepsilon}^{\varepsilon}|x|^{p+2}\nu_j(dx)}{\bar\sigma_j^2(\varepsilon)}\bigg)^{1/p}\bigg)\\&\quad +t\Big(\sqrt{\bar\sigma_1^2(\varepsilon)+\sigma_1^2}-\sqrt{\bar\sigma_2^2(\varepsilon)+\sigma_2^2}\Big)^2.
%&\leq \sum_{i=1}^2C\min\Big(2\sqrt t\bar\sigma_i(\varepsilon) , C'\varepsilon\Big)+t|\bar\sigma_1(\varepsilon)-\bar\sigma_2(\varepsilon)|.
 \end{align*}
\end{corollary}
\begin{proof}
This is a  consequence of Theorem \ref{teo:smalljumps} and Lemma \ref{lemma:gaussiane}.
\end{proof}

\subsection{Distances between random sums of random variables}\label{sss:bigjumps}

\begin{theorem}\label{teo:CPP}
  Let $(X_i)_{i\geq 1}$ and $(Y_i)_{i\geq 1}$ be sequences of i.i.d. random variables with $Y_i\in L_2$ and $N$, $N'$ be two positive integer-valued random variables with $N$ (resp. $N'$) independent
  of $(X_i)_{i\geq 1}$ (resp. $(Y_i)_{i\geq 1}$). Then, for $1\le p\le 2$,
\begin{align*}
  \Wi_p\bigg(\sum_{i=1}^{N} X_i, \sum_{i=1}^{N'} Y_i\bigg)&\leq \min\bigg( \big(c_p\E[N]Z_p(X_1,Y_1)\big)^{1/p},\E[N^{p}]^{1/p}\Wi_p(X_1,Y_1)\bigg)\\
  &\quad +
   \Wi_p(N,N')\E\big[|Y_1|^p\big]^{1/p}
\end{align*}
with the constant $c_p$ from Theorem \ref{teo:Rio09}.
 \end{theorem}

\begin{proof}
 By the triangle inequality,
 \begin{align}\label{eq:triangular}
 \Wi_p\bigg(\sum_{i=1}^N X_i, \sum_{i=1}^{N'} Y_i\bigg)\leq \Wi_p\bigg(\sum_{i=1}^{N} X_i, \sum_{i=1}^{N''} Y_i\bigg)+
 \Wi_p\bigg(\sum_{i=1}^{N''} Y_i, \sum_{i=1}^{N'} Y_i\bigg),
 \end{align}
 where $N''$ is independent of $(Y_i)_{i\geq 1}$ and with the same law as $N$.
 Thanks to Theorems \ref{teo:Rio09} and \ref{teo:zolotarevCPP}, the first summand in \eqref{eq:triangular} is bounded by
 $\big(c_p\E[N]Z_p(X_1,Y_1)\big)^{1/p}$. Alternatively, this summand can be estimated via Jensen's inequality joined with the fact that $\Wi_p\big(\sum_{i=1}^{N} X_i, \sum_{i=1}^{N''} Y_i\big)^p\leq \E\big[\big|\sum_{i=1}^{\tilde N} (X_i-Y_i)\big|^p\big]$, $\tilde N$ independent of $(X_i,Y_i)_{i\geq 1}$ and $\Li(\tilde N)=\Li(N)=\Li(N'')$, as follows:
 \[ \E\bigg[\bigg|\sum_{i=1}^{\tilde N} (X_i-Y_i)\bigg|^p\bigg] \le  \E\bigg[\tilde N^{p-1}\sum_{i=1}^{\tilde N} |X_i-Y_i|^p\bigg]\le  \E\big[N^{p}\big]\E\big[|X_1-Y_1|^p\big].
 \]
Therefore,
\begin{align*}
\Wi_p\bigg(\sum_{i=1}^{N} X_i, \sum_{i=1}^{N''} Y_i\bigg)^p&\leq \inf\bigg\{\E\bigg[\bigg|\sum_{i=1}^{\tilde N} (X_i-Y_i)\bigg|^p\bigg],
\, \tilde N\text{ independent of }(X_i,Y_i)_{i\ge 1}\bigg\}\\
 &\leq \E\big[N^{p}\big]\Wi_p(X_1,Y_1)^p.
\end{align*}
To control the second summand, we proceed similarly
\begin{align*}
&\Wi_p\bigg(\sum_{i=1}^{N''} Y_i, \sum_{i=1}^{N'} Y_i\bigg)^p\\
&\leq\inf\bigg\{\E\bigg[\bigg|\sum_{i=1}^{N''} Y_i-\sum_{i=1}^{N'} Y_i\bigg|^p\bigg],
 \, N'',N'\text{ independent of }(Y_i)_{i\ge 1}\bigg\}\\
% &\leq  \inf\bigg\{\E[|N-N'|^p]\E[|Y_1|^p], \Li(N)=\mu, \ \Li(N')=\nu\bigg\}\\
 &\leq\E[|Y_1|^p]\Wi_p(N'',N')^p,
\end{align*}
  which, by noting $\Li(N'')=\Li(N)$, concludes the proof.
 \end{proof}

In the preceding theorem one term is bounded alternatively by the Zolotarev or the Wasserstein distance  between $X_1$ and $Y_1$. The difference is the factor in front which is either the first or the $p$th moment of $N$. If $N$ is likely to be large, then better bounds can be obtained by profiting from the variance stabilisation for centred sums. Since the larger jumps are not our main issue, this is not pursued further.

In the Poisson case the moments and the Wasserstein distances can be easily analysed.

\begin{proposition}\label{prop:wpoisson}
 Let $N$ and $N'$ be two  Poisson random variables of mean  $\lambda$ and $\lambda'$, respectively. Let us denote by
 $m_{(p,\ell)}$ the moment of order $p$ of a Poisson random variable of mean $\ell$, i.e.
 $$m_{(p,\ell)}:=\sum_{i=1}^p\ell^i{p\brace i}, \text{ where } {p\brace i}:=\frac{1}{i!}\sum_{j=0}^{i}(-1)^{i-j}{i\choose j}j^p.$$
 Then the following upper
 bound holds for $p\geq1:$
 \begin{align*}
\Wi_p(N,N')^p\leq m_{(p,|\lambda-\lambda'|)}.
 \end{align*}
In particular,
\begin{align}
 \Wi_1(N,N')&\leq |\lambda-\lambda'| \label{eq:wass1poisson}\\
 \Wi_p(N,N')^p&\leq |\lambda-\lambda'|+ |\lambda-\lambda'|^p,\quad 1<p\le 2\label{eq:wass2poisson}.
\end{align}
 \end{proposition}

\begin{proof}
 Without loss of generality, let us suppose $\lambda'\geq \lambda$ and let $N''$ be a Poisson random variable with mean $\lambda-\lambda'$, independent of $N$. Thanks to Lemma \ref{lemma:proprieta} we have
 \begin{align*}
 \Wi_p(N,N')^p=\Wi_p(N'+N'',N')^p\leq \Wi_p(0,N'')^p\leq \E\big[(N'')^p\big]=m_{(p,\lambda'-\lambda)}.
 \end{align*}
To deduce \eqref{eq:wass1poisson} and \eqref{eq:wass2poisson} we use the fact that $m_{(1,\ell)}=\ell$, $m_{(2,\ell)}=\ell+\ell^2$ and $\E[(N'')^p]\le \E[N'']^{2-p}\E[(N'')^2]^{p-1}$ for $p\in(1,2]$ by H\"older's inequality.
 \end{proof}

 \subsection{First main result}
 %\subsection{Wasserstein distances between the increments of Lévy processes}

We will use the notation introduced in Section \ref{sec:notationlevy}. In accordance with that, for any given Lévy process $X^j$ with characteristics $(b_j,\sigma_j^2,\nu_j)$, $X^{j,B}(\varepsilon)$ will be a compound Poisson process with Lévy measure $\nu_j(dx)\I_{(\varepsilon,\infty)}(|x|)$, i.e.
$$X_t^{j,B}(\varepsilon)=\sum_{i=1}^{N_t^j}Y_i^{(j)}$$
where $N^j$ is a Poisson process of intensity $\lambda_j(\varepsilon):=\nu_j(\R\setminus (-\varepsilon,\varepsilon))$ independent of the sequence of i.i.d. random variables $(Y_i^{(j)})_{i\geq 0}$ having distribution $F_\varepsilon^j(dx)=\frac{\I_{(\varepsilon,\infty)}(|x|)}{\lambda_j(\varepsilon)}\nu_j(dx)$.
Recall from Proposition \ref{prop:wpoisson} that $m_{(p,\ell)}$ denotes the moment of order $p$ of a Poisson random variable of mean $\ell$.
\begin{theorem}\label{teow1}
Let $X^j$, $j=1,2$, be two Lévy processes with characteristics $(b_j,\sigma_j^2,\nu_j)$, $j=1,2$. For all $p\in[1,2]$, for all $\varepsilon\in [0,1]$ and all $t\geq 0$, the following estimate holds
% \begin{align*}
%\Wi_p\Big(P_t^{(b_1,\sigma_1,\nu_1)},P_t^{(b_2,\sigma_2,\nu_2)}\Big)&\leq \Wi_p
%\Big(\No\big(t b_1(\varepsilon),t\big(\sigma_1^2+\bar\sigma_1^2(\varepsilon)\big)\Big),
 %\No\Big(t b_2(\varepsilon),t\big(\sigma_2^2+\bar\sigma_2^2(\varepsilon)\big)\Big)\\
 % &\quad+
 %\sum_{j=1}^2\min\bigg(2\sqrt{t}\bar\sigma_j(\varepsilon),\frac{\int_{-\varepsilon}^{\varepsilon}|x|^3\nu_j(dx)}{2\bar\sigma^2_j(\varepsilon)}
 %\bigg)\\
 %&\quad +\Wi_p\Big(X_t^{1,B}(\varepsilon),X_t^{2,B}(\varepsilon)\Big)
 %\end{align*}
 %and
 \begin{align*}
\Wi_p\big(X_t^1,X_t^2\big)&\leq \Big(t^2\big(b_1(\varepsilon)-b_2(\varepsilon)\big)^2+t\big(\sigma_1+\bar\sigma_1(\varepsilon)-\sigma_2-\bar\sigma_2(\varepsilon)\big)^2\Big)^{1/2} \\
&\quad +C\sum_{j=1}^2\min\Big(\sqrt t\bar \sigma_j(\varepsilon) ,\varepsilon\Big)+
\Wi_p\big(X_{t}^{1,B}(\varepsilon),X_t^{2,B}(\varepsilon)\big),
 \end{align*}
 for some constant $C$, only depending on $p$.
Introducing $L_t(\varepsilon):=t|\lambda_1(\varepsilon)-\lambda_2(\varepsilon)|$, we have
\begin{align*}
\Wi_p\big(X_{t}^{1,B}(\varepsilon),X_t^{2,B}(\varepsilon)\big)&\leq
\big((t\lambda_1(\varepsilon))^{1/p}+t\lambda_1(\varepsilon)\big)\Wi_p\big(Y_1^{(1)},Y_1^{(2)}\big)\\
&\quad + \big(L_t(\varepsilon)^{1/p}+L_t(\varepsilon)\big) \E\big[\big(Y_1^{(2)}\big)^p\big]^{1/p}.
% t\lambda_1(\varepsilon)Z_p\big(Y_1^{(1)},Y_1^{(2)}\big)+C_p t^2\big(\var[Y_1^{(2)}]\big)^{p/2}\ell(\varepsilon)L(\varepsilon).
 \end{align*}
 \end{theorem}

\begin{proof}[Proof of Theorem \ref{teow1}]
By some abuse of notation let $\No(\mu,\sigma^2)$ denote a random variable with this distribution. Then, thanks to the Lévy-Itô decomposition, we have
 $$X_t^i=\No(t b_i(\varepsilon),t\sigma_i^2)+X_{t}^{i,S}(\varepsilon)+X_{t}^{i,B}(\varepsilon)$$
with independent summands.
 %with $\No(t b_i(\varepsilon),t\sigma_i^2)$, $X_{t}^{i,S}(\varepsilon)$, $X_{t}^{i,B}(\varepsilon)$ independent from each other.
 Hence, by subadditivity we get
 \begin{align*}
 \Wi_p\big(X^1_t,X^2_t\big)&\leq \Wi_p\big(\No(t b_1(\varepsilon),t\sigma_1^2)+X_{t}^{1,S}(\varepsilon),\No(t b_2(\varepsilon), t\sigma_2^2)+X_{t}^{2,S}(\varepsilon)\big)
 \\ &\quad+\Wi_p\big(X_{t}^{1,B}(\varepsilon), X_{t}^{2,B}(\varepsilon)\big).
 \end{align*}
Observe that
\begin{align*}
&\Wi_p\big(\No(t b_1(\varepsilon),t\sigma_1^2)+X_{t}^{1,S}(\varepsilon), \No(t b_2(\varepsilon),t\sigma_2^2)+
X_{t}^{2,S}(\varepsilon)\big)\\
&\leq \Wi_p\big(\No(t b_1(\varepsilon),t\sigma_1^2)+X_t^{1,S}(\varepsilon), \No(t b_1(\varepsilon),t(\sigma_1^2+\bar\sigma_1^2(\varepsilon)))\big)\\
&\quad + \Wi_p\big(\No(tb_2(\varepsilon),t\sigma_2^2)+X_{t}^{2,S}(\varepsilon), \No(t b_2(\varepsilon),t(\sigma_2^2+
\bar\sigma_2^2(\varepsilon)))\big)\\
&\quad +\Wi_p\big(\No(t b_1(\varepsilon),t(\sigma_1^2+\bar\sigma_2^2(\varepsilon))\big), \No(t b_2(\varepsilon),
t(\sigma_2^2+\bar\sigma_2^2(\varepsilon)))\big)\\
&\leq \Wi_p\big(X_{t}^{1,S}(\varepsilon),\No(0,t\sigma_1^2(\varepsilon))\big)
+\Wi_p\big(X_{t}^{2,S}(\varepsilon),\No(0,t\bar\sigma_2^2(\varepsilon))\big)\\
&\quad +\Wi_p\big(\No(t b_1(\varepsilon),t(\sigma_1^2+\bar\sigma_1^2(\varepsilon))),
 \No(t b_2(\varepsilon),t(\sigma_2^2+\bar\sigma_2^2(\varepsilon)))\big),
\end{align*}
where in the second inequality we used again Lemma \ref{lemma:proprieta}.
An application of Theorem \ref{teo:smalljumps} together with Point (2) in Lemma \ref{lemma:w1} and Lemma \ref{lemma:gaussiane} allows us to bound
$$\Wi_p\big(\No(t b_1(\varepsilon),t\sigma_1^2)+X_{t}^{1,S}(\varepsilon), \No(t b_2(\varepsilon),t\sigma_2^2)
+X_{t}^{2,S}(\varepsilon)\big)$$
by the quantity
 $$\Big(t^2\big(b_1(\varepsilon)-b_2(\varepsilon)\big)^2
 +t\big(\sigma_1+\bar\sigma_1(\varepsilon)-\sigma_2-\bar\sigma_2(\varepsilon)\big)^2\Big)^{1/2}
 +C\sum_{j=1}^2\min\Big(\sqrt t\bar \sigma_j(\varepsilon) ,\varepsilon\Big),$$
  for some constant $C$ only depending on $p$.
  Finally,  $\Wi_p\big(X_{t}^{1,B}(\varepsilon), X_{t}^{2,B}(\varepsilon)\big)$ is bounded by means of Theorem \ref{teo:CPP} and Proposition \ref{prop:wpoisson}.
 \end{proof}
%\begin{corollary}\label{cor:varepsilon}
%For all $\varepsilon\geq 0$ and all $t\geq 0$
 %\begin{align*}
%\Wi_1\Big(P_t^{(b_1,\sigma_1,\nu_1)},P_t^{(b_2,\sigma_2,\nu_2)}\Big)&\leq
%\sqrt{t^2\big(b_1(\varepsilon)-b_2(\varepsilon)\big)^2+t\big(\sigma_1+\bar\sigma_1(\varepsilon)-\sigma_2-\bar\sigma_2(\varepsilon)\big)^2} \\
%&\quad +\sum_{j=1}^2\min\Big(2\sqrt t\bar \sigma_j(\varepsilon) ,\frac{\varepsilon}{2}\Big)+
%\Wi_1\big(X_{t}^{1,B}(\varepsilon),X_t^{2,B}(\varepsilon)\big).
 %\end{align*}
%\end{corollary}

We now address the problem of how to compute the Wasserstein distance  between $n$ given increments of two Lévy processes.
To that end, fix a time span $T>0$,
a sample size $n\in\N$ and consider the sample $(X_{kT/n}^1-X_{(k-1)T/n}^1,X_{kT/n}^2-X_{(k-1)T/n}^2)_{k=1}^n$. From Lemma \ref{property:tensorization} we know that
we can measure the distance between the random vectors $(X_{kT/n}^1-X_{(k-1)T/n}^1)_{k=1}^n$ and
$(X_{kT/n}^2-X_{(k-1)T/n}^2)_{k=1}^n$ in terms of the Wasserstein distance between the marginals. This observation combined with Theorem \ref{teow1}
allows us to obtain an upper bound for the Wasserstein distance of order $p$ between the increments of these Lévy processes.

\begin{corollary}\label{res:tensorizationlevy}
 Let $X^j$, $j=1,2$, be two Lévy processes with characteristics $(b_j,\sigma_j^2,\nu_j)$, $j=1,2$. Then, with respect to the $\ell^r$-metric on $\R^n$ given by $d(x,y):=\big(\sum_{i=1}^N|x_i-y_i|^r\big)^{1/r}$, $r \geq 1$, for all $p\in[1,2]$,  $\varepsilon\geq 0$, $T>0$, $n\in\N$ we have
 \begin{align*}
  \Wi_p\Big(&(X_{kT/n}^1-X_{(k-1)T/n}^1)_{k=1}^n,(X_{kT/n}^2-X_{(k-1)T/n}^2)_{k=1}^n\Big)\\
  &\leq T n^{\frac{1}{r}-1} \big|b_1(\varepsilon)-b_2(\varepsilon)\big|
    + T^{1/2} n^{\frac{1}{r}-\frac12}\big|\sigma_1+\bar\sigma_1(\varepsilon)-\sigma_2-\bar\sigma_2(\varepsilon)\big|\\
&\quad  +
 C\sum_{j=1}^2\min\big(T^{1/2}n^{\frac{1}{r}-\frac12}\bar\sigma_j(\varepsilon),n^{\frac{1}{r}}\varepsilon\big)
 +n^{\frac{1}{r}}\Wi_p\big(X_{T/n}^{1,B}(\varepsilon),X_{T/n}^{2,B}(\varepsilon)\big),
 \end{align*}
 where $C$ is a constant depending only on $p$. The term $\Wi_p(X_{T/n}^{1,B}(\varepsilon),X_{T/n}^{2,B}(\varepsilon))$ can be bounded
 as in Theorem \ref{teo:CPP} with $t=T/n$.
\end{corollary}

In the Euclidean case $r=2$ we see that in the bound for the Wasserstein distance the drift part disappears as $n\to\infty$ ($T$ fixed), while the Gaussian part remains invariant and the Gaussian approximation of small jumps gives an error of order $\min(\bar\sigma_j(\varepsilon),n^{1/2}\varepsilon)$. The bound on the larger jumps scales as $n^{1/2}(T/n+(T/n)^{1/2})$ (for $p=1$ even as $T/n^{1/2}$) so that the entire bound on the Wasserstein distance remains bounded as $n\to\infty$.

 \subsection{Lower bounds}

Applying the general lower bound established in Proposition \ref{prop:toscani} to Lévy processes, we get the following result:

\begin{corollary}\label{cor:lb}
Let $W$ be a Brownian motion and $X^\varepsilon$ be a pure jump Lévy process with jumps of absolute value less than $\varepsilon\in(0,1]$. This means that $X^\varepsilon$ has L\'evy triplet $(0,0,\nu_\varepsilon)$, $\text{supp}(\nu_\varepsilon)\subset[-\varepsilon,\varepsilon]$ and
 characteristic function
$$\varphi^\varepsilon_t(u)=\E[e^{iuX_t^\varepsilon}]=\exp\Big(t \int_{-\varepsilon}^\varepsilon \big(e^{iux}-1-iux\big)\nu_\varepsilon(dx)\Big).$$
Let $\bar\sigma^2(\varepsilon)=\int x^2\nu_\varepsilon(dx)$. Then for $p\ge 1$
\begin{align*}
\Wi_p(X_t^\varepsilon,\bar\sigma(\varepsilon)W_t) &\geq  \sup_{u\in\R}\frac{\Big|\exp \Big(t\int\big(e^{iux}-1-iux\big)\nu_\varepsilon(dx)\Big)-\exp \Big(t\int \frac{(iux)^2}{2} \nu_\varepsilon(dx)\Big)\Big|}{\sqrt 2|u|}.
% &\geq  \sup_{u\in\R}\frac{\Big|\exp \Big(-t\int\big(1-\cos(ux)\big)\nu_\varepsilon(dx)\Big)-\exp \Big(-t\int \frac{(ux)^2}{2} \nu_\varepsilon(dx)\Big)\Big|}{\sqrt 2|u|}.
 \end{align*}
\end{corollary}

By the bound given in Proposition \ref{prop:toscani} we usually do not lose in approximation order as the following lower bound examples demonstrate.

Let us start with the general case that at $\varepsilon=1$ we have a standardised pure jump process $X^1$ with $\E[X^1_t]=0$, $\var[X^1_t]=t$ as for Brownian motion, which means $\int x\nu_1(dx)=0$, $\int x^2\nu_1(dx)=1$. Then rescaling as in Donsker's Theorem we consider $X_t^\varepsilon:=\varepsilon X^1_{\varepsilon^{-2}t}$ such that $\nu_\varepsilon(B)=\varepsilon^{-2}\nu_1(\varepsilon^{-1}B)$ for Borel sets $B$, $\bar\sigma^2(\varepsilon)=1$  and
$$\varphi_t^\varepsilon(u)=\exp\Big(t\varepsilon^{-2} \int_{-1}^1\big(e^{i\varepsilon ux}-1\big)\nu_1(dx)\Big).$$
Let us further assume that $q_3:=\int x^3\nu_1(dx)\not=0$. Then, taking into account the first two moments, a Taylor expansion yields
\[ t\varepsilon^{-2}\int_{-1}^1\big(e^{i\varepsilon ux}-1\big)\nu_1(dx)=-\frac{tu^2}{2}-\frac{it\varepsilon u^3}{3!}q_3+{ O}(t\varepsilon^2 u^4).
\]
For $t>\varepsilon^2$ we thus obtain at $u_0=t^{-1/2}$
\[ \frac{|\varphi_\varepsilon(u_0)-e^{-tu_0^2/2}|}{u_0}=t^{1/2}e^{-1/2}\Big|e^{-iq_3\varepsilon t^{-1/2}/3!+{ O}(\varepsilon^2/t)}-1\Big| =\varepsilon\Big(\frac{q_3}{\sqrt e 3!}+{ O}(\varepsilon/\sqrt{t})\Big).
\]
Hence, by Corollary \ref{cor:lb} there are constants $M>0$ and $c>0$ such that for all $t\ge M\varepsilon^2$
\[ \Wi_p(X_t^\varepsilon,W_t) \geq c\varepsilon.\]
For $t\le M\varepsilon^2$ we  obtain at $u_0=(8M\lambda_1/t)^{1/2}$ with $\lambda_1=\nu_1(\R)$
\begin{align*}
 \frac{|\varphi_t^\varepsilon(u_0)-e^{-tu_0^2/2}|}{u_0}&\ge \frac{t^{1/2}}{(8M\lambda_1)^{1/2}}\Big|e^{t\varepsilon^{-2}\int(\cos(\varepsilon u_0x)-1)\nu_1(dx)}-e^{-4M\lambda_1}\Big|
\\ &\ge \frac{e^{-2M\lambda_1}-e^{-4M\lambda_1}}{(8M\lambda_1)^{1/2}}t^{1/2}.
\end{align*}
We  conclude
\[\forall t>0,\,\varepsilon\in(0,1]:\;\Wi_p(X_t^\varepsilon,W_t)\geq K\min(\sqrt t,\varepsilon)\]
for some positive constant $K$, depending on $\nu_1$, but independent of $t$ and $\varepsilon$, whenever $q_3=\int x^3\nu_1(dx)\not=0$.

Even for symmetric L\'evy measures, not inducing skewness of the distribution, we can attain the order $\min(\sqrt t,\varepsilon)$. If we consider $\nu_1=\frac12(\delta_{-1}+\delta_1)$, $\nu_\varepsilon=\frac{\delta_{-\varepsilon}+\delta_{\varepsilon}}{2\varepsilon^2}$, we arrive at the same conclusion by computing the distance $T_1$ between $Y_t(\varepsilon)$ and $\No(0,t)$ as in Remark \ref{rmk:ex}:
 \begin{align*}
  T_1\big(\Li(Y_t(\varepsilon)),\No(0,t)\big)
  =\sup_{u\in\R}\bigg|\frac{\exp\Big(-t\Big(\frac{1-\cos(u\varepsilon)}{\varepsilon^2}\Big)\Big)-\exp\Big(-\frac{tu^2}{2}\Big)}{u}\bigg|.
 \end{align*}
If $t\geq  \varepsilon^2$, we choose $u=\frac{2\pi}{\varepsilon}$ and get
 \begin{align*}
  T_1\big(\Li(Y_t(\varepsilon)),\No(0,t)\big)\geq \bigg|\frac{\varepsilon}{2\pi}\Big(1-\exp\Big(-\frac{2\pi^2t}{\varepsilon^2}\Big)\Big)\bigg|\geq \Big(\frac{1-e^{-2\pi^2}}{2\pi}\Big)\varepsilon.
 \end{align*}
If $t<\varepsilon^2$, the choice $u=\frac{3}{\sqrt t}$ gives
$$  T_1\big(\Li(Y_t(\varepsilon)),\No(0,t)\big)\geq \sqrt t\frac{\Big(e^{-2}-e^{-\frac{9}{2}}\Big)}{3}.$$

We conclude that also in this case $\Wi_1(\Li(Y_t(\varepsilon)),\No(0,t))\geq K\min(\sqrt t,\varepsilon)$ holds for some positive constant $K$, independent of $t$ and $\varepsilon$.

\section{Total variation bounds via convolution}\label{sec:conv}

\subsection{Notation and some useful properties}

Let $(\X,\mathscr F)$ be a measurable space and let $\mu$ and $\nu$ be two probability measures on $(\X,\mathscr F)$.
\begin{definition}
 The \emph{total variation distance} between $\mu$ and $\nu$ is defined as
 $$\|\mu-\nu\|_{TV}=\sup_{A\in\F}\big|\mu(A)-\nu(A)\big|.$$
\end{definition}

\begin{lemma}\label{property:tv}
 The total variation distance has the following properties.
 \begin{enumerate}
% \item $\|\mu-\nu\|_{TV}=\frac{1}{2}L_1(\mu,\nu)$ (one can always take $\frac{\mu+\nu}{2}$ as a dominating measure).
  \item $\|\mu-\nu\|_{TV}=\frac{1}{2}\sup_{\|\Psi\|_{\infty}\leq 1}\big|\int_\X \Psi(x)(\mu-\nu)(dx)\big|$.
 % the supremum being taken over all compactly supported $\psi$ satisfying condition $\sup_{x\in\X}|\Psi(x)|\leq 1$.
  \item \label{pro:couplingtv} $\|\mu-\nu\|_{TV}=\inf\big(\p(X\neq Y):\Li(X)=\mu,\ \Li(Y)=\nu\big)$.
 \end{enumerate}
\end{lemma}

\begin{remark}
 Let $\X$ be a discrete set, equipped with the Hamming metric $d(x,y)=\I_{x\neq y}$. In this case, thanks to Property \ref{pro:couplingtv}. above,
 for any probability measures $\mu$ and $\nu$ on $\X$ we have
 $$\Wi_1(\mu,\nu)=\|\mu-\nu\|_{TV}.$$
\end{remark}

The total variation distance does not always bound the Wasserstein distance, because the latter is also influenced by large distances.
However, thanks to the following classical result, one can get some control on $\Wi_p$ given a bound on the total variation distance.
\begin{theorem}\label{teo:villani}[See \cite{villani09}, Theorem 6.13]
 Let $\mu$ and $\nu$ be two probability measures on a Polish space $(\X,d)$. Let $p\in[1,\infty)$ and $x_0\in\X$. Then
 $$\Wi_p(\mu,\nu)\leq 2^{\frac{1}{p'}}\bigg(\int d(x_0,x)^p|\mu-\nu|(dx)\bigg)^{\frac{1}{p}},\quad \frac{1}{p}+\frac{1}{p'}=1.$$
 In particular, if $p=1$ and the diameter of $\X$ is bounded by $D$, then
 $$\Wi_1(\mu,\nu)\leq 2 D\|\mu-\nu\|_{TV}.$$
\end{theorem}
In Proposition \ref{prop:TV} we will show an inequality that can be thought of as an inverse of the one above. Namely, the total variation distance between  two measures convolved with a common measure can be bounded by a multiple of the Wasserstein distance of order $1$.

\subsection{Wasserstein distance of order $1$ and total variation distance}

Recall that a real function $g$ is of \emph{bounded variation} if its total variation norm is finite, i.e.
$$\|g\|_{BV}=\sup_{ P \in\mathscr P}\sum_{i=0}^{n_P-1}|g(x_{i+1})-g(x_i)|<\infty,$$
where the supremum is taken over the set $\mathscr{P}=\{P=(x_0,\dots,x_{n_P}):x_0<x_1<\dots<x_{n_P}\}$ of all finite ordered subsets of $\R$. We will denote by $BV(\R)$ the space of functions of bounded variation.

We now state a lemma that will be useful in the following.

\begin{lemma}\label{lemma:bv}
 Let $g$ be a real function of bounded variation and $\mathcal{F} \subseteq \{\phi \colon \R \to \R  : \|\phi\|_\infty \leq 1\} \cap L^{1}(\R)$ a functional class. Suppose that for any $\phi\in \mathcal{F}$
 $$h_\phi(t)=\int_\R \phi(y)\Big(g(t-y)-\lim_{x\to -\infty} g(x)\Big)dy$$
is well defined. Then,
 \begin{equation}\label{eq:Lipnorm}
  \sup_{\phi \in \mathcal{F}}\|h_\phi\|_{Lip}\leq \|g\|_{BV}.
 \end{equation}
 
\end{lemma}
\begin{proof}
The proof is an easy consequence of the following classical results on Lebesgue-Stieltjes measures:
\begin{enumerate}
\item For every right-continuous function $g\colon \R\to\R$ of bounded variation there exists a unique signed measure $\mu$ such that
\begin{equation}\label{eq:signedmeasure}
\mu(]-\infty,x])=g(x)-\lim_{y\to-\infty} g(y).
\end{equation}
\item Let $\phi\in L^\infty(\R)$ and let $g\in BV(\R)$ be a right-continuous function. Let $\mu$ be the finite signed measure associated to $g$ as in \eqref{eq:signedmeasure}.Then $\int \phi(t-y)\mu(dy)$ is well defined, measurable in $t\in\R$ and bounded in absolute value by $\|\phi\|_\infty\|g\|_{BV}$. \label{eq:2}
\end{enumerate}
More precisely, let $\mu$ be the finite signed measure associated to $g$. It is enough to prove that $\int \phi(t-y)\mu(dy)$ is the weak derivative of $h_\phi$ since then, using Point \ref{eq:2}. above, we deduce that $\|h_\phi\|_{Lip}=\|\int \phi(\cdot-y)\mu(dy)\|_\infty\leq \|\phi\|_\infty\|g\|_{BV}$ and hence \eqref{eq:Lipnorm}. The claim above follows by Fubini's Theorem: for all $T>0$ \begin{align*}
\int_0^T \int \phi(t-y)\mu(dy)dt &= \int\int \I_{[0,T]}(u+y)\phi(u)du\,\mu(dy)\\
&=\int \phi(u)(g(T-u)-g(-u))du\\
&=\int \phi(u)\Big(g(t-u)-\lim_{x\to-\infty} g(x)\Big)du\Big|_{t=0}^T.
\end{align*}
Hence, $ \int \phi(t-y)\mu(dy)$ is the weak derivative of $\int \phi(u)(g(t-u)-\lim_{x\to-\infty} g(x))du$ as desired.
\end{proof}

\begin{proposition}\label{prop:TV}
 Let $\mu$ and $\nu$ be two measures on $(\R,\B(\R))$ and $G$ be an absolutely continuous measure with respect to the Lebesgue measure admitting
 a density $g$ of  bounded variation.
Then the total variation distance between the convolution measures $\mu*G$ and $\nu*G$ is bounded by
$$\|\mu*G-\nu*G\|_{TV}\leq \frac{\|g\|_{BV}}{2} \Wi_1(\mu,\nu).$$
\end{proposition}
\begin{proof}
 \begin{align*}
  \|\mu*G-\nu*G\|_{TV}&=\frac{1}{2}\sup_{\|\phi\|_{\infty}\leq 1}\bigg|\int_\R \phi(x)(\mu*G-\nu*G)(dx)\bigg|\\
                      &=\frac{1}{2}\sup_{\|\phi\|_{\infty}\leq 1}\bigg|\int_\R \bigg(\int_\R \phi(x)g(x-t)(\mu-\nu)(dt)\bigg)dx\bigg|\\
                      &=\frac{1}{2}\sup_{\|\phi\|_{\infty}\leq 1}\bigg|\int_\R \bigg(\int_\R \phi(x)g(x-t)dx\bigg)(\mu -\nu)(dt)\bigg|,
 \end{align*}
the supremum being taken over compactly supported functions $\phi$.
Denote by $h_\phi(t)=\int_\R \phi(x)g(x-t)dx$. From the last equality it follows that
$$\|\mu*G-\nu*G\|_{TV}\leq \frac{1}{2}\sup_{\|\phi\|_{\infty}\leq 1}\sup_{\|\psi \|_{Lip}\leq \|h_\phi\|_{Lip}}\bigg|\int_\R \psi(t)(\mu -\nu)(dt)\bigg|,$$
hence, applying Lemma \ref{lemma:bv} to $\mathcal F=\{\phi \colon \R \to \R  : \|\phi\|_\infty \leq 1 \textnormal{ with compact support}\}$ and Proposition \ref{pro:lip}, we deduce that
\begin{align*}
 \|\mu*G-\nu*G\|_{TV}\leq \frac{\|g\|_{BV}}{2}\sup_{\|\psi\|_{Lip}\leq 1}\bigg|\int_\R \psi(t)(\mu -\nu)(dt)\bigg|=\frac{\|g\|_{BV}}{2}\Wi_1(\mu,\nu).
\end{align*}
\end{proof}
The upper bound established in Proposition \ref{prop:TV} is sharp. To see that, let us consider the following example.

\begin{example}
 Let $\mu=\delta_0$, $\nu=\delta_{\varepsilon}$ and $G=\mathcal N(0,1)$ for some $\varepsilon>0$. Denoting by $\varphi$ the density of a random variable $N\sim\No(0,1)$ and by $\Phi$ its cumulative distribution function, we have
\begin{align*}
 \|\nu*G-\mu*G\|_{TV}&=\frac{1}{2}\int_\R|\varphi(x)-\varphi(x-\varepsilon)|dx = \Phi\Big(\frac{\varepsilon}{2}\Big) - \Phi\Big(-\frac{\varepsilon}{2}\Big) = 2\Phi\Big(\frac{\varepsilon}{2}\Big)-1\\
 &= \frac{\varepsilon}{\sqrt{2\pi}} + O(\varepsilon^2).
\end{align*}
At the same time it is easy to see that $\Wi_1(\mu,\nu)=\varepsilon$ and $\|g\|_{BV}=\sqrt \frac{2}{\pi}$. Therefore, the upper bound established in Proposition \ref{prop:TV}
 $$ \|\nu*G-\mu*G\|_{TV}\leq \frac{1}{\sqrt{2\pi}}\Wi_1(\mu,\nu)=\frac{\varepsilon}{\sqrt{2\pi}}$$
  is exactly the correct estimate up to the first order.
\end{example}

\subsection{Total variation distance and Toscani-Fourier distances}

For any Lebesgue density  $f$ introduce its Fourier transform $\F f(u)=\int e^{iux}f(x)dx$.
A first elementary result linking the total variation distance between convolution measures to Toscani-Fourier metrics is the following.
\begin{proposition}\label{prop:CS}
Let $\mu,\nu$ and $G$ be probability measures and suppose that its characteristic functions $\varphi_\mu$, $\varphi_\nu$, $\varphi_G$ are differentiable. Assume that $G$ has a Lebesgue density $g$ with $m$th weak derivative $g^{(m)}$. Then, for all $k,j,r\in\{1,\ldots,m\}$, we have
\begin{align*}
&\|\mu*G-\nu*G\|_{TV}\\
&\leq C\bigg(T_{k}(\mu,\nu) \|g^{(k)}\|_2+\sqrt 2T_{r}(\mu,\nu) \|(xg(x))^{(r)}\|_2+\sqrt 2\sup_{u\in\R}\frac{|\varphi'_\mu(u)-\varphi'_\nu(u)|}{|u|^{j}}\|g^{(j)}\|_2\bigg)
\end{align*}
for some numerical constant $C>0.$
\end{proposition}
\begin{proof}
First of all, remark that if any one among the $\|g^{(\bullet)}\|_2$, $\|(xg(x))^{(\bullet)}\|_2$, $T_{\bullet}(\mu,\nu)$, or $\sup_{u\in\R}\frac{|\varphi'_\mu(u)-\varphi'_\nu(u)|}{|u|^{\bullet}}$ appearing above is infinite, then there is nothing to prove. Therefore, from now on, we will assume that they are all finite.
Since $G$ admits a density $g$ with respect to Lebesgue measure, $\mu*G$ and $\nu*G$ have densities  $g*\mu$ and $g*\nu$.

Using the Cauchy-Schwarz inequality we have
\begin{align*}
\|\mu*G-\nu*G\|_{TV}&=\frac{1}{2}\int \frac{1}{\sqrt{1+x^2}}\sqrt{1+x^2}|g*\mu(x)-g*\nu(x)|dx\\
&\leq C(\|g*\mu-g*\nu\|_2+\|x(g*\mu-g*\nu)\|_2),
\end{align*}
for some numerical constant $C>0.$
For all $k>0$ an application of the Plancherel identity yields
\begin{align*}	
	\|g*\mu-g*\nu\|_2^2&=\frac{1}{2\pi}\|\varphi_G(\varphi_\mu-\varphi_\nu)\|_2^2=\frac{1}{2\pi}\int \frac{|\varphi_\mu(u)-\varphi_\nu(u)|^2}{u^k}	|\varphi_G(u)|^2 u^k du.
\end{align*}
Hence,
$$\|g*\mu-g*\nu\|_2\leq \sqrt{\frac{1}{2\pi}}\sup_{u\in\R}\frac{|\varphi_\mu(u)-\varphi_\nu(u)|}{|u|^{k/2}}\|u^{k/2}\varphi_G\|_2.$$
In the same way we also have
$$\|x(g*\mu(x)-g*\nu(x))\|_2^2\leq \frac{1}{\pi}\|\varphi_G' (\varphi_\mu-\varphi_\nu)\|_2^2+\frac{1}{\pi}\|\varphi_G (\varphi_\mu'-\varphi_\nu')\|_2^2$$
and we conclude as before that for all $r,j>0$
\begin{align*}
&\|x(g*\mu(x)-g*\nu(x))\|_2\\
&\quad \leq\sqrt{\frac{1}{\pi}} \sup_{u\in\R}\frac{|\varphi_\mu(u)-\varphi_\nu(u)|}{|u|^{r/2}}\|u^{r/2}\varphi_G'\|_2
+\sqrt{\frac{1}{\pi}}\sup_{u\in\R}\frac{|\varphi'_\mu(u)-\varphi'_\nu(u)|}{|u|^{j/2}}\|u^{j/2}\varphi_G\|_2.
\end{align*}
It remains to apply the inverse Fourier transform.
\end{proof}

Using a different set of hypotheses, one can also establish the following relation between the total variation distance and the Toscani-Fourier distance.

\begin{proposition}\label{prop:N}
Let $\mu$, $\nu$ and $G$ be real probability measures absolutely continuous with respect to the Lebesgue measure. Let $f_\mu$, $f_\nu$ and $g$ denote their densities and  $F_\mu$ and $F_\nu$ denote the cumulative distribution functions of $\mu$ and $\nu$. Suppose that $\F g \in L_1$  and that $F_\mu - F_\nu \in L_1$. Further suppose that the graphs of $f_\mu*g$ and $f_\nu*g$ intersect in at most $N$ points. Then,
$$\|\mu*G-\nu*G\|_{TV}\leq \frac{N}{2\pi}T_1(\mu,\nu)\int|\F g(u)|du.$$
\end{proposition}
\begin{proof}
As in the proof of Proposition \ref{prop:TV}, let us introduce the function
$$h_\phi(t):=\int_{\R}\phi(x)g(t-x)dx$$
and recall that
$$\|\mu*G-\nu*G\|_{TV}=\frac{1}{2}\sup_{\|\phi\|_{\infty}\leq 1}\bigg|\int_\R h_\phi(t)(\mu-\nu)(dt) \bigg|.$$
Using an integration by part and Plancherel identity, we get
\begin{align}
\int_\R h_\phi(t)(\mu-\nu)(dt) &=\int_\R h_{\phi}'(t)(F_\mu(t)-F_\nu(t))dt\nonumber\\&=\frac{1}{2\pi}\int \overline{\F h_{\phi}'(u)}\F(F_\mu-F_\nu)(u) du\nonumber \\
&=\frac{1}{2\pi}\int \overline{\F h_{\phi}'(u)}\frac{\varphi_\mu(u)-\varphi_\nu(u)}{-iu} du\nonumber\\
&\leq T_1(\mu,\nu) \frac{1}{2\pi}\int |\F h_{\phi}'(u)|du. \label{eq:TVT}
\end{align}
Also observe that
$$\sup_{\|\phi\|_{\infty}\leq 1}\bigg|\int_\R h_\phi(t)(\mu-\nu)(dt) \bigg|=\int_\R h_{\tilde\phi}(t)(\mu-\nu)(dt)$$
where
$$\tilde \phi(x)=\begin{cases}
-1\quad & \text{ if }\quad f_\mu*g(x)<f_\nu*g(x),\\
1\quad & \text{ if }\quad f_\mu*g(x)\geq f_\nu*g(x).\\
\end{cases}$$
Let us denote by $-\infty=x_0<x_1<\dots<x_N<x_{N+1}=+\infty$ the points of intersections between the graphs of $f_\mu*g$ and $f_\nu*g$. In particular $\tilde\phi(u)=\pm\sum_{i=0}^N(-1)^i\I_{[x_i,x_{i+1})}(u)$ with the sign depending on the sign of $f_{\mu}*g-f_\nu*g$ on $(-\infty,x_1)$. Thus,
$$h_{\tilde \phi}'(u)=\pm2\sum_{j=1}^N (-1)^{j}g(u-x_j).$$
In particular, we get that
$$\F h_{\tilde \phi}'(u)=\pm 2\sum_{j=1}^N (-1)^{j}\F g(u)e^{iux_j},$$
hence $|\F h_{\tilde \phi}'(u)|\leq 2N  |\F g(u)|.$ This fact, together with \eqref{eq:TVT}, concludes the proof.
\end{proof}
Let us observe that another way to link the total variation distance between convolution measures to the Toscani-Fourier distance is offered by Theorem 2.21 in \cite{CT07} joint with Proposition \ref{prop:TV}. More precisely, Theorem 2.21 in \cite{CT07} states that, under appropriate hypotheses on $\mu$ and $\nu$,
$$\Wi_1(\mu,\nu)\leq \bigg(\frac{18 M}{\pi}\bigg)^{1/3}T_2(\mu,\nu)^{1/6},$$
with
$M=\max\big\{\E[X^2],\E[Y^2]\big\},$ $X\sim \mu$ and $Y\sim \nu.$
%$$M=\max\bigg\{\int v^2\mu(dv),\int v^2\nu(dv)\bigg\}.$$
Therefore, from Proposition \ref{prop:TV}, it follows that
$$\|\mu*G-\nu*G\|_{TV}\leq \|g\|_{BV}\bigg(\frac{9M}{4\pi}\bigg)^{1/3}T_2(\mu,\nu)^{1/6},$$
where $g$ denotes the density of $G$.
Using some ideas from the proof of Theorem 2.21 in \cite{CT07} we will be able to prove the following general result.
 \begin{proposition}\label{prop:CTMR}
Let $\mu,\nu\in\mathcal P_j(\R)$, $j\geq 1$, and $G$ be a measure, absolutely continuous with respect to the Lebesgue measure. Suppose that the density $g$ of $G$ is $j$-times weakly differentiable with $j$th derivative $g^{(j)}\in L_2$. Then,
$$\|\mu*G-\nu*G\|_{TV}\le C_j^{1/(2j+1)}\| g^{(j)}\|_2^{2j/(2j+1)} T_j(\mu,\nu)^{2j/(2j+1)},$$
where $C_j=\max\big(\E\big[|X+Z|^j\big],\E\big[|Y+Z|^j\big]\big)$ with $X\sim \mu$, $Y\sim \nu$, $Z\sim G$ and $Z$ independent of $X$ and $Y$.
\end{proposition}

\begin{proof}
Using the same notation as in Proposition \ref{prop:N}, we have for all $R>0$
\begin{align*}
&\|\mu*G-\nu*G\|_{TV}=\frac{1}{2}\int |g*\mu(x)-g*\nu(x)|dx\\&\leq \frac{1}{2}\bigg(\int_{-R}^{R} |g*\mu(x)-g*\nu(x)|dx+\frac{1}{R^j}\int_{|x|>R}|x|^j|g*\mu(x)-g*\nu(x)|dx\bigg)\\
&\leq \frac{1}{2}\bigg(\int_{-R}^{R} |g*\mu(x)-g*\nu(x)|dx+\frac{C_j}{R^j}\bigg).
\end{align*}
By Cauchy-Schwarz inequality,
$$\int_{-R}^{R} |g*\mu(x)-g*\nu(x)|dx\leq \sqrt{2R} \|g*\mu-g*\nu\|_2$$
holds. Taking $R=\big(\frac{ C_j}{\sqrt 2\|g*\mu-g*\nu\|_2}\big)^{2/(2j+1)}$ we get
$$\|\mu*G-\nu*G\|_{TV}\le \frac{1}{2} C_j^{1/(2j+1)}(\sqrt 2\|g*\mu-g*\nu\|_2)^{2j/(2j+1)}.$$
Using Plancherel identity and the properties of the Fourier transform, we deduce that
\begin{align*}
\|g*\mu-g*\nu\|_2^2&=\frac{1}{2\pi}\|\F(g*\mu)-\F(g*\nu)\|_2^2
=\frac{1}{2\pi}\|\F g(\varphi_{\mu}-\varphi_{\nu})\|_2^2\\
&=\frac{1}{2\pi}\int |\F g(u)|^2 u^{2j}\frac{|\varphi_{\mu}(u)-\varphi_{\nu}(u)|^2}{u^{2j}}du\\
&\leq \frac{1}{2\pi}T_j(\mu,\nu)^2 \int |\F g(u)|^2 u^{2j}du\\
&= \frac{1}{2\pi}T_j^2(\mu,\nu)\|\F g^{(j)}\|_2^2= T_j^2(\mu,\nu)\| g^{(j)}\|_2^2.
\end{align*}
It follows that
$$\|\mu*G-\nu*G\|_{TV}\le C_j^{1/(2j+1)}(T_j(\mu,\nu)\| g^{(j)}\|_2)^{2j/(2j+1)}.$$
\end{proof}

\begin{remark}
To better understand the upper bounds presented above, let us specialise to the case $G=\mathcal N(0,\sigma^2)$. In order to compare the results presented in Propositions \ref{prop:CS}--\ref{prop:CTMR} let us start by observing that the following equalities hold.
\begin{itemize}
\item If $g(x)=\frac{1}{\sqrt{2\pi}\sigma}e^{-\frac{x^2}{2\sigma^2}}$, then $\F g(u)=e^{-\frac{u^2\sigma^2}{2}}$. Therefore, $\int|\F g(u)|du=\frac{\sqrt{2\pi}}{\sigma}.$
\item Also, $g'(x)=\frac{-x}{\sqrt{2\pi}\sigma^3}e^{-\frac{x^2}{2\sigma^2}}$ and $g''(x)=\frac{1}{\sqrt{2\pi}\sigma^3}e^{-\frac{x^2}{2\sigma^2}}(-1+\frac{x^2}{\sigma^2})$. It follows that $\|g'\|_2^2=\frac{1}{4\sqrt{\pi}\sigma^3}$, $\|g''\|_2^2=\frac{1}{\pi\sigma^5}\int_0^\infty e^{-y^2}(y^2-1)^2 dy=\frac{1}{4\sqrt \pi\sigma^5}$ and $\|g\|_{BV}=\sqrt{\frac{2}{\sigma^2\pi}}$.
\end{itemize}
We are now able to compare the previous results for independent random variables $Z\sim\mathcal N(0,\sigma^2)$, $X\sim \mu$, $Y\sim \nu$.

\begin{description}

\item {\it Proposition \ref{prop:CS} for $T_1$:} With a numerical constant $C>0$, independent of the laws of $X,Y,Z$,
\begin{multline*}
\|\Li(X+Z)-\Li(Y+Z)\|_{TV}\leq C\bigg( T_1(X,Y)\bigg(\frac{1}{\sqrt{\sigma^3}}+\frac{1}{\sigma}\bigg)\\+\frac{1}{\sqrt{\sigma^3}}\sup_{u\in\R}\frac{|\varphi_X'(u)-\varphi_Y'(u)|}{|u|}\bigg).
\end{multline*}
\item {\it Proposition \ref{prop:N}:} Let $N$ be the number of intersections between the graphs of the densities of $X+Z$ and $Y+Z$. Then,
$$\|\Li(X+Z)-\Li(Y+Z)\|_{TV}\leq \frac{N}{\sqrt{2\pi}}\frac{T_1(X,Y)}{\sigma}.$$

\item {\it Proposition \ref{prop:CTMR} for $T_1$:}
\begin{multline*}
\|\Li(X+Z)-\Li(Y+Z)\|_{TV}\\\leq \Big(\frac{\max\big(\E[|X+Z|],\E[|Y+Z|]\big)}{4\sqrt\pi}\Big)^{1/3}\frac{\big(T_1(X,Y)\big)^{2/3}}{\sigma}.
\end{multline*}

\item {\it Proposition \ref{prop:CTMR}  for $T_2$:}
\begin{multline*}
\|\Li(X+Z)-\Li(Y+Z)\|_{TV}\\ \leq \bigg(\frac{\max\big(\E\big[(X+Z)^2\big],\E\big[(Y+Z)^2\big]\big)}{16\pi}\bigg)^{1/5}\frac{(T_2(X,Y))^{4/5}}{\sigma^2}.
\end{multline*}

\item {\it Proposition \ref{prop:TV} + Theorem 2.21 in \cite{CT07}:}
$$\|\Li(X+Z)-\Li(Y+Z)\|_{TV}\leq \bigg( \frac{9\max\big(\E[X^2],\E[Y^2]\big)}{\sqrt{2\pi^5}}\bigg)^{1/3}\frac{(T_2(X,Y))^{1/6}}{\sigma}.$$
\end{description}

We see that Proposition \ref{prop:CTMR} gives a much tighter bound than Proposition \ref{prop:TV} + Theorem 2.21 in \cite{CT07} when $T_2(X,Y)$ is small.
\end{remark}

\subsection{Main total variation results}\label{subsec:TV}
As it was the case in Section \ref{sec:uppb}, in order to obtain an upper bound for the total variation distance between the marginals $X_t^1$ and $X_t^2$ of two Lévy processes it is enough to separetely control the total variation distance between Gaussian distributions, between the small jumps and the corresponding Gaussian component and finally between the big jumps. The latter can be controlled by means of the following result.

 \begin{theorem}\label{teo:CPPTV}
  Let $(X_i)_{i\geq 1}$ and $(Y_i)_{i\geq1}$ be sequences of i.i.d. random variables a.s. different from zero and $N$, $N'$ be two Poisson random variables with $N$ (resp. $N'$) independent of $(X_i)_{i\geq 1}$ (resp. $(Y_i)_{i\geq 1}$). Denote by $\lambda$ (resp. $\lambda'$) the mean of $N$ (resp. $N'$). Then,
    $$\Big\|\Li\Big(\sum_{i=1}^N X_i\Big)-\Li\Big(\sum_{i=1}^{N'}Y_i\Big)\Big\|_{TV}\leq (\lambda\wedge\lambda')\|\Li(X_1)-\Li(Y_1)\|_{TV}+1-e^{-|\lambda-\lambda'|}.$$
  \end{theorem}

  \begin{proof}
  Without loss of generality, let us suppose that $\lambda\geq \lambda'$ and write $\lambda=\alpha+\lambda'$, $\alpha\geq 0$.
  By  triangle inequality,
  \begin{align}
  \Big\|\Li\Big(\sum_{i=1}^N X_i\Big)-\Li\Big(\sum_{i=1}^{N'}Y_i\Big)\Big\|_{TV} &\leq \Big\|\Li\Big(\sum_{i=1}^N X_i\Big)-\Li\Big(\sum_{i=1}^{N''}X_i\Big)\Big\|_{TV}\nonumber\\
  &\quad + \Big\|\Li\Big(\sum_{i=1}^{N''} X_i\Big)-\Li\Big(\sum_{i=1}^{N'}Y_i\Big)\Big\|_{TV},\label{eq:ti}
  \end{align}
  where $N''$ is a random variable independent of $(X_i)_{i\geq 1}$ and with the same law as $N'$.
  The first addendum in \eqref{eq:ti} can be bounded as follows. Let $P$ be a Poisson random variable independent of $N''$ and $(X_i)_{i\geq 1}$ with mean $\alpha$. Then,
  \begin{align*}\Big\|\Li\Big(\sum_{i=1}^N X_i\Big)-\Li\Big(\sum_{i=1}^{N''}X_i\Big)\Big\|_{TV} &=\Big\|\Li\Big(\sum_{i=1}^{N''+P} X_i\Big)-\Li\Big(\sum_{i=1}^{N'}X_i\Big)\Big\|_{TV}\\ &\leq \Big\|\delta_0-\Li\Big(\sum_{i=1}^{P}X_i\Big)\Big\|_{TV}
  \end{align*}
  where the last bound follows by subadditivity of the total variation distance.
  By definition, it is easy to see that
  $$\Big\|\delta_0-\Li\Big(\sum_{i=1}^{P}X_i\Big)\Big\|_{TV}=  \p\Big(\sum_{i=1}^{P}X_i\neq 0\Big)\leq 1-e^{-\alpha}.$$
  In order to bound the second addendum in \eqref{eq:ti} we condition on $N'$ and use again the subadditivity of the total variation joined with the fact that $\Li(N')=\Li(N'')$:
  \begin{align*}
  \Big\|\Li\Big(\sum_{i=1}^{N''} X_i\Big)-\Li\Big(\sum_{i=1}^{N'}Y_i\Big)\Big\|_{TV}&=\sum_{n\geq 0} \Big\|\Li\Big(\sum_{i=1}^{n} X_i\Big)-\Li\Big(\sum_{i=1}^{n}Y_i\Big)\Big\|_{TV}\p(N'=n)\\&\leq \sum_{n\geq0} n \|\Li(X_1)-\Li(Y_1)\|_{TV}\p(N'=n)\\&=\lambda'\|\Li(X_1)-\Li(Y_1)\|_{TV}.
  \end{align*}
  \end{proof}

The treatment of the small jumps is the subject of the following result:

\begin{proposition}\label{prop:TVprocesses}
 Let $X$ be a pure jump Lévy process with Lévy measure $\nu$. Introduce $\nu_\varepsilon=\nu\I_{|x|\leq \varepsilon}$. Then, for all $\Sigma>0$ and $\varepsilon\in(0,1]$, we have
 \begin{align*}
  \Big\|P_t^{(0,\Sigma,\nu_\varepsilon)}-P_t^{(0,\sqrt{\Sigma^2+\bar\sigma^2(\varepsilon)},0)}\Big\|_{TV}&\leq
  \sqrt{\frac{2}{\pi t\Sigma^2}}\Wi_1\Big(P_t^{(0,0,\nu_\varepsilon)},P_t^{(0,\bar\sigma(\varepsilon),0)}\Big)\\
  &\leq \sqrt{\frac{2}{\pi t\Sigma^2}}\min\Big(2\sqrt{t\bar\sigma^2(\varepsilon)},\frac{\varepsilon}{2}\Big).
 \end{align*}
%The  $\Sigma^2\geq \frac{2}{\pi t}$ is asked in order to guarantee that $\int_\R|g'(x)|dx\leq 1$ where $g$ denotes the density of
%a Gaussian random variable $\Nn(0,t\Sigma^2)$ with respect to the Lebesgue measure.
\end{proposition}

\begin{proof}
This follows by applying first Proposition \ref{prop:TV} and then Theorem \ref{teo:smalljumps}.
\end{proof}

As a consequence of the above estimates on the Wasserstein distances, we obtain a bound for the total variation distance of the marginals of L\'evy processes with non-zero Gaussian components.

\begin{theorem}\label{th:MainTV}
 With the same notation used in Theorem \ref{teow1} and Section \ref{sec:notationlevy}, for all $t>0$, $\varepsilon\in (0,1]$ and for all $\sigma_i>0$, $i=1,2$, we have:
 \begin{align*}
\Big\|P_t^{(b_1,\sigma_1,\nu_1)}&-P_t^{(b_2,\sigma_2,\nu_2)}\Big\|_{TV} \\&\leq  \frac{\sqrt{\frac{t}{2\pi}} \Big|b_1(\varepsilon)-b_2(\varepsilon)\Big|+\sqrt 2\Big|\sqrt{\sigma_1^2+\bar\sigma_1^2(\varepsilon)}-\sqrt{\sigma_2^2+\bar\sigma_2^2(\varepsilon)}\Big|}{\sqrt{\sigma_1^2+\bar\sigma_1^2(\varepsilon)}\vee \sqrt{\sigma_2^2+\bar\sigma_2^2(\varepsilon)}}\\
&\quad+ \sum_{i=1}^2\sqrt{\frac{2}{\pi t \sigma_i^2}}
 \min\Big(2\sqrt{t\bar \sigma_i^2(\varepsilon)}, \frac{\varepsilon}{2}\Big)\\
 &\quad
 +t\big|\lambda_1(\varepsilon)-\lambda_2(\varepsilon)\big|
 +t\big(\lambda_1(\varepsilon)\wedge\lambda_2(\varepsilon)\big)
 \big\|\frac{\nu_1^\varepsilon}{\lambda_1(\varepsilon)}-\frac{\nu_2^\varepsilon}{\lambda_2(\varepsilon)}\big\|_{TV},
 \end{align*}
with $\nu_j^\varepsilon=\nu_j(\cdot\cap (\R\setminus(-\varepsilon,\varepsilon)))$ and $\lambda_j(\varepsilon)=\nu_j^\varepsilon(\R)$.

\end{theorem}

\begin{proof}
 By subadditivity of the total variation distance and by the triangle inequality,
 \begin{align*}
  \Big\|P_t^{(b_1,\sigma_1,\nu_1)}-P_t^{(b_2,\sigma_2,\nu_2)}\Big\|_{TV}& \leq \sum_{i=1}^2\Big\|P_t^{(0,\sigma_i,\nu_i(\varepsilon))}-P_t^{(0,\sqrt{\sigma_i^2+\bar\sigma_i^2(\varepsilon)},0)}\Big\|_{TV}\\
  &\quad +\Big\|P_t^{(b_1(\varepsilon),\sqrt{\sigma_1^2+\bar\sigma_1^2(\varepsilon)},0)}-P_t^{(b_2(\varepsilon),\sqrt{\sigma_2^2+\bar\sigma_2^2(\varepsilon)},0)}\Big\|_{TV}\\
  &\quad + \big\|\Li\big(X_t^{1,B}(\varepsilon)\big)-\Li\big(X_t^{2,B}(\varepsilon)\big)\big\|_{TV}.
 \end{align*}
The proof follows from Proposition \ref{prop:TVprocesses}, the classical bound $$\|\No(\mu_1,\sigma_1^2)-\No(\mu_2,\sigma_2^2)\|_{TV}\leq \frac{\frac{1}{\sqrt{2\pi}}|\mu_1-\mu_2|+\sqrt 2|\sigma_1-\sigma_2|}{\sigma_1\vee\sigma_2}.
$$
and   Theorem \ref{teo:CPPTV}.
\end{proof}

Another useful result follows directly from Proposition \ref{prop:CTMR} with $j=1$ and allows to bound the total variation distance for L\'evy processes with positive Gaussian part by the Toscani-Fourier distance for the same L\'evy processes, but with a smaller Gaussian part.

\begin{theorem}\label{teo:TVToscani}
Let $X^i\sim(b_i,\sigma_i^2,\nu_i)$, with $\sigma_i>0$, $i=1,2$, be two Lévy processes. For any $\Sigma\in(0,\sigma_1\wedge\sigma_2)$  consider the Lévy processes $\widetilde X^i\sim(b_i,\sigma_i^2-\Sigma^2,\nu_i)$, $i=1,2$. Then,
$$\big\|\Li(X_t^1)-\Li(X_t^2)\big\|_{TV}\leq \frac{\max\big(\E[| X_t^1|], \E[| X_t^2|]\big)^{1/3} \big(T_1(\widetilde X_t^1,\widetilde X_t^2)\big)^{2/3}}{( 16\pi)^{1/6}\Sigma\sqrt t}.
$$
\end{theorem}

\section{A statistical application}

\subsection{Lower bounds in the minimax sense}\label{subsec:minimax}
One of the main goals in statistics is to estimate a quantity of interest from the data. There are different criteria that can be used to judge the quality of an estimator. In nonparametric statistics it is common to use a minimax approach. Let us recall the classical setting. From the data $(X_1,\dots,X_n)$ one wants to recover a quantity of interest $\theta$ (e.g. $\theta$ is the density of the observations, or the regression function, or the Léyy density, or the diffusion coefficient, etc.). In practice $\theta$ is unknown (but supposed to belong to a certain \emph{parameter space} $\Theta$) and one needs to estimate it via an estimator (a measurable function of the data) $\hat \theta_n=\hat \theta_n(X_1,\dots,X_n)$. To measure the accuracy of the estimator one computes the \emph{minimax risk}
$$\mathcal R_n^*:=\inf_{T_n}\sup_{\theta\in\Theta}\E\big[d^2(\theta,T_n)\big],$$
where the infimum is taken over all possible estimators $T_n$ of $\theta$ and $d$ is a semi-distance on $\Theta$. Furthermore, one says that a positive sequence $(\psi_n)_{n\geq 1}$ is an \emph{optimal rate of convergence} of estimators on $(\Theta,d)$ if there exist constants $C<\infty$ and $c>0$ such that
\begin{equation}\label{eq:ub}
\limsup_{n\to\infty}\psi_n^{-2}\mathcal R_n^*\leq C\quad \quad \text{                  (upper bound)}
\end{equation}
and
\begin{equation}\label{eq:lb}
\liminf_{n\to\infty}\psi_n^{-2}\mathcal R_n^*\geq c, \quad \quad  \text{                 (lower bound)}.
\end{equation}
The goal is then to construct an estimator $\theta_n^*$ such that
$$\sup_{\theta\in\Theta} \E\big[d^2(\theta_n^*,\theta)\big]\leq C'\psi_n^2$$
where $(\psi_n)_{n\geq 1}$ is the optimal rate of convergence and $C'<\infty$ is a constant.

The usual way to proceed is to build an estimator $\hat\theta_n$ of $\theta$ and start the investigation about its performance firstly via an upper bound like \eqref{eq:ub}. This is important since the first thing to check is that the considered estimator is at least \emph{consistent}, that is automatically implied if $\sup_{\theta\in\Theta}\E\big[d^2(\theta,\hat \theta_n)\big]\to0$. After that, a natural question is whether one could construct a better (in terms of rate of convergence in the class $(\Theta,d)$) estimator. In order to ensure that it is not possible to obtain a better estimator than the one already constructed one has to prove a lower bound, that is it is needed to prove that the rate of convergence of any other possible estimator of $\theta$ will not be faster than the rate obtained in the upper bound. This is in general a difficult task and we refer to Chapter 2 in \cite{T} for general techniques to prove lower bounds. Without recalling all the steps needed to prove a lower bound following \cite{T}, let us stress here that one of the fundamental ingredients is to have a fine upper bound for the total variation distance or other measure distances. To that end the estimates in Section \ref{subsec:TV} can be of general interest to prove lower bounds in the minimax sense.

One situation when this general procedure applies is the following, where we show how to simplify the arguments used in \cite{JR} in order to prove the desired lower bound for an estimator of the integrated volatility.

\subsection{How to simplify the proof of the lower bound in \cite{JR}}\label{subsec:JR}

In \cite{JR}, the authors consider a one-dimensional Itô-semimartingale
$X$ with characteristics $(B,C,\nu)$:
$$B_t=\int_0^t b_sds,\quad C_t=\int_0^t c_sds,\quad \nu(dt,dx)=dtF_t(dx).$$
They assume that $X$ belongs to the class $\mathcal S_A^r$ of all Itô-semimartingales that satisfy
$$|b_t|+c_t+\int (|x|^r\wedge 1)F_t(dx)\leq A \quad \forall t\in[0,1].$$
Their goal is to estimate the integrated volatility $C$ at time $1$, $C(X)_1$, from high-frequency observations $X_{\frac{i}{n}}$, $i=0,\dots,n$. They have an upper bound for an estimator of $C(X)_1$ and they want to prove that the rate of convergence attained by that estimator is optimal. To that aim they need to prove that any uniform rate $\psi_n$ for estimating $C(X)_1$ satisfies
\begin{equation}\label{eq:JRlb}
\psi_n\geq (n\log n)^{-\frac{2-r}{2}}\quad \text{if }\ r>1.
\end{equation}
Following \cite{T}, their strategy consists in finding two Lévy processes $X^i\sim(b_i,\sigma_i^2,F_i)$, $i=1,2$, such that
\begin{enumerate}
\item $\sigma_1^2-\sigma_2^2=a_n:=(n\log n)^{-\frac{2-r}{2}}$, $r\in(0,2)$,
\item $\int (|x|^r\wedge 1)F_i(dx)\leq K$,
\item $\|\mathscr L((X_{i/n}^1)_{1\leq i\leq n})-\mathscr L((X_{i/n}^2)_{1\leq i\leq n})\|_{TV}\to 0$ as $n\to \infty$.
\end{enumerate}
The construction in \cite{JR} of the Lévy processes as well as the proof of the convergence in total variation stated in Point 3. above is very involved. Let us now see how to use Theorem \ref{teo:TVToscani} to prove \eqref{eq:JRlb}  more easily.
To that aim consider two sequences of Lévy processes $X^{1,n}\sim(0,1+a_n,F^1_{n})$ and $X^{2,n}\sim(0,1,F^2_{n})$
with Lévy measures $F^1_n$ and $F^2_n$ satisfying the following conditions:
\begin{itemize}
\item $\int_{\R}(|x|^r\wedge 1)F^i_n(dx)\leq K$, $i=1,2.$
\item Define
$\Psi_{i,n}:=\int_{\R}(e^{iux}-1-iux\I_{|x|\leq 1})F^i_n(dx)$, $i=1,2.$
Then $\Psi_{1,n}$ and $\Psi_{2,n}$ are real positive functions such that
\begin{equation}\label{eq:psi}
\Psi_{2,n}(u)=\frac{a_n}{2} + \Psi_{1,n}(u), \quad \forall |u|<u_n:=2\sqrt{n\log n}.
\end{equation}
\end{itemize}
It is not difficult to see that Lévy measures $F^1_n$ and $F^2_n$ satisfying such conditions always exist.
In particular, it follows from \eqref{eq:psi} that the $X^{i,n}$, $i=1,2$, have the same characteristic function for all $|u|<u_n$, i.e.:
\begin{align*}
\E\Big[e^{iuX^{1,n}_{1/n}}\Big]&=\exp\Big(-\frac{u^2}{2n}(1+a_n)-\frac{\Psi_{1,n}(u)}{n}\Big),\\
\E\Big[e^{iuX^{2,n}_{1/n}}\Big]&=\exp\Big(-\frac{u^2}{2n}-\frac{\Psi_{2,n}(u)}{n}\Big).
\end{align*}
In order to apply Theorem \ref{teo:TVToscani} let us observe that $X^{1,n}_{1/n}$ (resp. $X^{2,n}_{1/n}$) is equal in law to the convolution between a Gaussian distribution $\mathcal N\big(0,\frac{1}{8n}\big)$ and $\widetilde X^{1,n}_{1/n}$ (resp. $\widetilde X^{2,n}_{1/n}$), where $\widetilde X^{1,n}\sim\big(0,\frac{7}{8}+a_n, F^1_n\big)$ (resp. $\widetilde X^{2,n}\sim\big(0,\frac{7}{8}, F^2_n\big)$).
We  obtain
%Recall that $\mathscr L(X_{1/n})=\mu_n*G_n$ and $\mathscr L(Y_{1/n})=\nu_n*G_n$.
\begin{align*}
\|\mathscr L(X^{1,n}_{1/n})-\mathscr L(X^{2,n}_{1/n})\|_{TV}\leq \Big(\frac{32}{\pi}\Big)^{1/6} \big(C_n n^{3/4}T_1(\mu_n, \nu_n)\big)^{2/3},
\end{align*}
where $C_n^2=\max\big(\E[| X^{1,n}_{1/n}|],\E[| X^{2,n}_{1/n}|]\big)$.
We are therefore left to compute $T_1(\mu_n, \nu_n)$ and show that $n\|\mathscr L(X_{1/n}^{1,n})-\mathscr L(X_{1/n}^{2,n})\|_{TV}\to 0$.
\begin{align*}
&T_1(\mu_n, \nu_n)=\sup_{u\in\R}\frac{\exp\big(-\frac{u^2}{2n}(\frac{7}{8}+a_n)-\frac{\Psi_{1,n}(u)}{n}\big)-\exp\big(-\frac{7u^2}{16n}-\frac{\Psi_{2,n}(u)}{n}\big)}{u}\\
&=\sup_{|u|>u_n}\frac{\exp\big(-\frac{7u^2}{16n})\Big[\exp\big(-\frac{u^2a_n}{2n}-\frac{\Psi_{1,n}(u)}{n}\big)-\exp\big(-\frac{\Psi_{2,n}(u)}{n}\big)\Big]}{u}\\
&\leq \frac{\exp\big(-\frac{7u_n^2}{16n})}{u_n}=\frac{\exp\big(-\frac{7\times4n\log n}{16n})}{2\sqrt{n\log n}}=\frac{n^{-9/4}}{2\sqrt{\log n}}.
\end{align*}
Hence,
$$\|\mu_n*G_n-\nu_n*G_n\|_{TV}\leq \bigg(C_n n^{3/4}\frac{n^{-9/4}}{\sqrt{\log n}}\bigg)^{2/3}=\Big(\frac{32}{\pi}\Big)^{1/6}\frac{C_n^{2/3} n^{-1}}{(\log n)^{1/3}}.$$
Therefore,
\begin{align*}
\|\mathscr L((X_{i/n}^{1,n})_{1\leq i\leq n})-\mathscr L((X_{i/n}^{2,n})_{1\leq i\leq n})\|_{TV}&\leq \sqrt{n\|\mathscr L(X^{1,n}_{1/n})-\mathscr L(X^{2,n}_{1/n})\|_{TV}}
\\&\leq \Big(\frac{32}{\pi}\Big)^{1/12}\Big(\frac{C_n}{\sqrt {\log n}}\Big)^{1/3}\to 0,
\end{align*}
as desired.

\providecommand{\bysame}{\leavevmode\hbox to3em{\hrulefill}\thinspace}
\providecommand{\MR}{\relax\ifhmode\unskip\space\fi MR }
\providecommand{\MRhref}[2]{%
  \href{http://www.ams.org/mathscinet-getitem?mr=#1}{#2}
}
\providecommand{\href}[2]{#2}

\end{document}